%% file: Proprietes_eparses_vf.tex
\documentclass[11pt]{amsart} 

\usepackage[utf8]{inputenc}
\usepackage[T1]{fontenc}
\usepackage[french]{babel}

\usepackage[bottom=1in]{geometry}
\usepackage{esint} 
\usepackage{graphicx, amssymb}
\usepackage[mathscr]{euscript}
  \usepackage{todonotes}

\usepackage[breaklinks]{hyperref}

 \usepackage{enumitem}

\renewcommand{\and}{et}  
\newtheorem{theorem}{Th\'eor\`eme}
\newtheorem{definition}[theorem]{D\'efinition}
\newcommand{\ssize}{\text{size}\,}

\newtheorem{lemma}[theorem]{Lemme}
\newtheorem{corollary}[theorem]{Corollaire}
\newtheorem{proposition}[theorem]{Proposition}
\newtheorem*{remark}{Remarque}

\newtheorem*{assum-L}{Hypoth\`eses sur $L$}
\DeclareMathOperator{\dive}{div}

\newcommand{\sssize}{\widetilde{\text{size}\,}}
\newcommand{\size}{{\text{size}\,}}

\newcommand{\one}{\mathbf{1}}

\newcommand{\rr}{\mathbb}
\newcommand{\ii}{\mathscr}
\newcommand{\ic}{\mathcal}
\newcommand{\ci}{\tilde{\chi}}
\newcommand{\loc}{\text{loc}}
\newcommand{\ds}{\displaystyle}

\newtheorem{question*}{Question}
\newtheorem*{main*}{\underline{Induction statement}}

\newcommand{\osc}{{\textrm osc}}
\newcommand{\mcS}{{\mathcal S}}
\newcommand{\BMO}{{\textrm BMO}}
\newcommand{\Dy}{{\mathbb D}}
\newcommand\E{\mathbb{E}}

\def\Xint#1{\mathchoice
   {\XXint\displaystyle\textstyle{#1}}%
   {\XXint\textstyle\scriptstyle{#1}}%
   {\XXint\scriptstyle\scriptscriptstyle{#1}}%
   {\XXint\scriptscriptstyle\scriptscriptstyle{#1}}%
   \!\int}
\def\XXint#1#2#3{{\setbox0=\hbox{$#1{#2#3}{\int}$}
     \vcenter{\hbox{$#2#3$}}\kern-.5\wd0}}

\def\aver#1{\Xint-_{#1}}

\title[Contr\^ole \'epars d'un op\'erateur et applications]{Conservation de certaines propri\'et\'es \`a travers un contr\^ole \'epars d'un op\'erateur et applications au projecteur de Leray-Hopf}

\author{Cristina Benea \& Fr\'{e}d\'{e}ric Bernicot}
\address{Cristina Benea \& Fr\'{e}d\'{e}ric Bernicot, CNRS - Universit\'{e} de Nantes, Laboratoire Jean Leray, Nantes 44322, France}
\email{cristina.benea@univ-nantes.fr \& frederic.bernicot@univ-nantes.fr}

\date{\today}


\thanks{Les auteurs sont support\'es par le projet ERC FAnFArE no. $637510$}

\begin{document}

\begin{abstract} Nous poursuivons l'\'etude d'un contr\^ole \'epars d'un op\'erateur singulier. Plus pr\'ecis\'ement nous expliquons comment on peut conserver certaines propri\'et\'es de l'op\'erateur initial \`a travers un tel contr\^ole et d\'ecrivons quelques applications: bornitude de l'adjoint de la transform\'ee de Riesz et du projecteur de Leray. De plus, nous nous int\'eresserons \`a donner un regard nouveau sur les dominations \'eparses \`a travers les oscillations et les fonctions carr\'ees localis\'ees. Aussi, nous d\'evoilerons une connexion entre les bons intervalles de la d\'ecomposition \'eparse et une d\'ecomposition atomique.
\end{abstract}

\maketitle

\begin{quote}
\footnotesize\tableofcontents
\end{quote}

\section{Introduction}

Ce travail a pour but de poursuivre l'\'etude d'un contr\^ole \'epars d'un op\'erateur, et de comprendre comment on peut conserver certaines propri\'et\'es au travers d'un tel contr\^ole.
Rappelons tout d'abord qu'une collection de cubes $\mcS:=(Q)_{Q\in \mcS}$ de l'espace Euclidien ${\rr R}^n$ est dite {\it \'eparse} si il existe une collection d'ensembles majeurs disjoints (voir Definition \ref{def:sparse}).  Associ\'e \`a une telle collection, on peut consid\'erer l'op\'erateur \'epars d\'efini par:
$$ T_{\mcS}(f) := x \mapsto  \sum_{Q\in \mcS} \left( \aver{Q} f(z) \, dz\right) {\bf 1}_{Q}(x)   , \qquad \forall f\in L^1_{loc}.$$
L'\'etude des op\'erateurs \'epars a commenc\'e dans \cite{Lernerfirst,CUMP} o\`u il est d\'emontr\'e que ces op\'erateurs v\'erifient des estimations \`a poids $L^p(\omega)$ optimales pour des poids de Muckenhoupt $\omega \in A_p$.

 R\'ecemment, ces op\'erateurs ont connu un tr\`es grand int\'eret, puisqu'ils permettent d'avoir une preuve relativement simple du th\'eor\`eme $A_2$ (Lerner \cite{Lerner1,Lerner2,Lerner-simplerA_2} et Lacey \cite{LaceyA_2}) et de nombreux travaux ont contribu\'e au d\'eveloppement de cette approche: contr\^ole ponctuel d'un op\'erateur de Calder\'on-Zygmund par un op\'erateur \'epars, contr\^ole de diff\'erents op\'erateurs singuliers (transform\'ees de Riesz associ\'ees \`a un semi-groupe \cite{BernicotFreyPetermichl}, Bochner-Riesz \cite{BeneaBernicotLuque}, BHT \cite{weighted_BHT}, op\'erateurs singuliers homog\`enes \cite{CCDO}, ...) \`a travers une forme bilin\'eaire (ou trilin\'eaire) \'eparse ... 

La domination \'eparse d'un op\'erateur lin\'eaire $T$ ou de sa forme bilin\'eaire $\Lambda_T= (f,g) \mapsto \langle Tf,g \rangle$ s'est r\'ev\'el\'ee utile pour montrer des estimations aux poids optimales: en effet ``{\it le caract\`ere \'epars laisse de la place au poids pour n'intervenir qu'une seule fois}''. Un aspect important vient  aussi de l'op\'erateur maximal \`a poids qui est born\'e sur des espaces $L^p$ \`a poids de mani\`ere ind\'ependante de la caract\'eristique du poids.  

En g\'en\'eral, le mieux qu'on peut esp\'erer pour la forme bilin\'eaire d'un op\'erateur de Calder\'on-Zygmund, est la domination \'eparse (dans le sens que la collection $\ic S$ est \'eparse) avec des moyennes dans $L^1$:
\begin{equation}
\vert  \langle Tf, g  \rangle \vert \lesssim  \sum_{Q \in \ic S}\aver{Q}\vert f \vert dx \cdot \aver{Q}\vert g \vert dx \cdot \vert Q \vert. \label{eq:epars}
\end{equation}

\medskip

La philosophie d'un contr\^ole \'epars d'un op\'erateur est la suivante: {\it dominer un op\'erateur singulier par une repr\'esentation g\'eom\'etrique (sur une collection d'intervalles ou cubes \'eparses, c'est \`a dire 'rar\'efi\'e') \`a l'aide d'op\'erateurs \'el\'ementaires positifs (localis\'es \`a l'\'echelle des diff\'erents cubes) qui sont facilement estimables (dans le cas ci-dessus, des op\'erateurs de moyennes dans $L^1$).}

Dans le membre de droite du contr\^ole \'epars, on a une expression qui semble avoir perdu les propri\'et\'es de l'op\'erateur de Calder\'on-Zygmund initial, mais on a une forme bilin\'eaire qui implique la bornitude $L^p \to L^p$ de l'op\'erateur $T$ ainsi que des estimations \`a poids. Dans ce travail, nous souhaitons comprendre quelles sont les informations suppl\'ementaires que l'on peut conserver au travers d'un contr\^ole \'epars et nous d\'ecrirons plusieurs applications.

\medskip

Avant d'expliciter les diff\'erentes propri\'et\'es que nous allons \'etudier, nous souhaitons motiver l'obtention d'un contr\^ole \'epars, autrement que par l'obtention des estimations \`a poids:  un op\'erateur (ou une forme bilin\'eaire) \'epars est un op\'erateur {\it auto-adjoint} (positif), qui v\'erifie de bonnes estimations (quantifi\'ees) dans des espaces \`a poids, ainsi que des estimations de type faible $L^1$ (ce qu'on peut montrer en utilisant une d\'ecomposition de Calder\'on-Zygmund (sic!)). Par cons\'equent, si un op\'erateur admet un contr\^ole \'epars, alors son adjoint est de type faible $L^1$ (voir Corollaire \ref{coro:typefaible}). Nous appliquerons ce principe dans le cadre des transform\'ees de Riesz, fonctionnelles carr\'es, et projecteur de Leray associ\'es \`a un op\'erateur elliptique du second ordre $\dive(A\nabla \cdot)$, obtenant de nouvelles estimations de type faible $L^1$ pour ces op\'erateurs (voir Section \ref{sec:leray}).

\medskip

\large{ \bf Estimation \'eparse et r\'egularit\'e.}
Les sections \ref{sec:T1=0} et \ref{sec:gradient} sont d\'edi\'ees \`a l'\'etude de propri\'et\'es de r\'egularit\'e, et comment on peut les conserver via un contr\^ole \'epars. En particulier, nous obtiendrons d'abord une repr\'esentation \'eparse d'un op\'erateur de Calder\'on-Zygmund v\'erifiant la condition $T(1)=0$ en termes d'oscillation. Ces oscillations permettent de conserver le caract\`ere $T(1)=0$ \`a l'\'echelle \'el\'ementaire de la repr\'esentation \'eparse. Une telle am\'elioration nous permettra de faire quelques observations sur la composition d'op\'erateurs de Calder\'on-Zygmund. \\
Puis en Section \ref{sec:gradient}, nous nous int\'eresserons au gradient d'un op\'erateur de Calder\'on-Zygmund r\'egulier, et nous verrons que l'on peut, en un certain sens, commuter l'op\'erateur gradient avec la propri\'et\'e de repr\'esentation \'eparse. Ceci nous permettra d'en d\'eduire des estimations optimales pour un tel op\'erateur dans des espaces de Sobolev \`a poids.

\medskip

\large{ \bf Estimation \'eparse pour un multiplicateur de Fourier/Haar via un principe de localisation.}
Le fait qu'un op\'erateur $T$ born\'e sur $L^2$ puisse admettre un contr\^ole \'epars comme \eqref{eq:epars} est un peu surprenant: si on utilise la bornitude $L^2 \to L^2$ de l'op\'erateur, on  obtient
\[
\vert \langle T f, g  \rangle \vert \lesssim \|f\|_2 \cdot \|g\|_2,
\] 
et si on localise ce r\'esultat, on peut dire que la version localis\'ee sur un cube $Q$ de $\Lambda_{T_Q}$ est born\'ee par des moyennes en $L^2$:
\begin{equation}
\label{eq:Holder-L2}
\vert \langle T_{Q} f, g  \rangle \vert \lesssim \big(\aver{Q}\vert f \vert^2 dx\big)^{\frac{1}{2}} \cdot \big(\aver{Q}\vert g \vert^2 dx\big)^{\frac{1}{2}} \cdot \vert Q \vert.
\end{equation}
Mais en effet, l'id\'ee du contr\^ole \'epars est de dire qu'en choisissant les bons cubes sur lesquels on d\'ecompose, on peut avoir des moyennes $L^1$.

Dans la section \ref{sec:haar}, nous allons nous sp\'ecialiser \`a un cadre o\`u l'op\'erateur $T$ (et sa forme bilin\'eaire) peuvent facilement \^etre localis\'es. Cette situation est le cas des multiplicateurs de Fourier, ou de Haar. Nous verrons alors comment peut-on tirer avantage et conserver la structure {fr\'equentielle} de l'op\'erateur initial \`a travers une estimation \'eparse. \\
En effet, par exemple si $T$ admet une repr\'esentation simple par des ondelettes, alors \eqref{eq:Holder-L2} est remplac\'ee par la m\^eme estimation, mais avec des moyennes $L^2$ de la fonction carr\'ee (localis\'ee) de $f$ et de $g$, respectivement. Apr\`es, on peut se servir du th\'eor\`eme de John-Nirenberg pour obtenir des moyennes dans $L^{1, \infty}$ de la fonction carr\'ee, chacune d'elles pr\'ec\'ed\'ee par un `$\sup$' sur une collection d'intervalles (ou cubes) dyadiques. Avec un temps d'arr\^et soigneusement construit, on peut \`a la fin obtenir une domination \'eparse, avec des moyennes $L^1$ des fonctions $f$ et $g$.  

Dans l'analyse de temps-fr\'equence, on utilise souvent des structures contenues dans le plan espace-fr\'equence, et l'information sur la fonction $f$ est conserv\'ee dans deux quantit\'es: le `size' et l'`energy'. L'\'energie concerne plut\^ot l'orthogonalit\'e des paquets d'onde, et pour un op\'erateur de Fourier classique, on peut le transformer en `size' (qui est une moyenne maximale en $L^1$ ou $L^{1, \infty}$), en localisant sur des bons intervalles. Cette id\'ee de transformer les moyennes $L^2$ en moyennes $L^1$ apparait souvent dans ce secteur de l'analyse.

De plus, la localisation, comme present\'ee dans la Proposition \ref{prop:local-Hilbert}, est d\'ej\`a apparue dans \cite{vv_BHT}, o\`u les auteurs emploient une telle localisation en espace (construite par un temps d'arr\^et gouvern\'e par les ensembles de niveau du `size') afin de montrer des estimations \`a valeurs vectorielles pour les paraproduits et pour la transform\'e bilin\'eaire de Hilbert $BHT$.

Observons que dans le cas le plus facile, d'un multiplicateur de Haar, la moyenne $L^2$ de la fonction carr\'ee de $f$ correspond exactement \`a l'oscillation $L^2$ de $f$, et il est alors plus \'evident qu'on doit utiliser le th\'eor\`eme de John-Nirenberg pour obtenir l'oscillation $L^1$. 

Ces id\'ees sont pr\'esent\'ees dans section \ref{sec:haar}. Plus pr\'ecis\'ement, dans la sous-section \ref{subsec:p>1}, on expose ces id\'ees pour obtenir une estimation \'eparse, permettant de retrouver les estimations \`a poids dans $L^p(\omega)$ pour $p>1$. Dans la sous-section \ref{subsec:carre}, on d\'ecrit comment on peut conserver l'information fr\'equentielle, \`a l'aide des fonctions carr\'ees localis\'ees dans la repr\'esentation \'eparse. Le th\'eor\`eme de John-Nirenberg, dans la forme de \cite{multilinear_harmonic}, qu'on peut g\'en\'eraliser pour des espaces \`a poids (d'une mani\`ere ind\'ependante du poids), nous permet d'employer n'importe quelle moyenne $L^p$ de la fonction carr\'ee, avec $0<p <\infty$. Par cons\'equent, en utilisant un temps d'arr\^et similaire \`a celui pour la d\'ecomposition atomique \'eparse,  on en d\'eduit alors dans la sous-section \ref{subsec:p<1}, des estimations dans des espaces de Hardy \`a poids $H^p(\omega)$ pour $p\in(0,1]$:
\[
T: H^p_\omega \to H^p_\omega, \quad \text{ pour tout poids    } \omega \quad \text{et pour tout  }  0<p \leq 1,
\] 
avec une norme ind\'ependante du poids (am\'eliorant ainsi des r\'esultats de Lee,Lin \cite{LeeLin} et Lee,Lin,Yang \cite{LeeLinYang})!

\medskip

\large{ \bf Application \`a la transform\'ee de Riesz et au projecteur de Leray.}
La section \ref{sec:leray} est d\'edi\'ee \`a des applications dans le cadre des transform\'ees de Riesz, associ\'ees \`a un op\'erateur elliptique sous forme divergence $L=-\dive A\nabla$. Plus pr\'ecis\'ement, nous verrons que sous l'hypoth\`ese que cet op\'erateur g\'en\`ere un semi-groupe, dont le gradient admet des estimations Gaussiennes ponctuelles, alors la transform\'ee de Riesz $R_L:=\nabla L^{-1/2}$  a un comportement vraiment similaire \`a celui des op\'erateurs de Calder\'on-Zygmund. En particulier, nous obtiendrons les nouveaux r\'esultats suivants
\begin{itemize}
\item[$\bullet$] L'op\'erateur dual $(R_L)^*$ est de type faible $L^1$, ainsi que diff\'erentes fonctionnelles carr\'ees.
\item[$\bullet$] Le projecteur de Leray $\pi_L:=R_L (R_{L^*})^*$ est aussi de type faible $L^1$ et admet des estimations $L^p(\omega)$ \`a poids optimales si $p\in(1,\infty)$ et $H^p(\omega)$ si $p\in(0,1]$. 
\end{itemize}

\medskip

\large{ \bf Lien entre contr\^ole \'epars et oscillation.}
D'autre part, dans le cadre des fonctions de Haar, on obtient (sous-section \ref{sec:atomic-dec}) une d\'ecomposition atomique et simultan\'ement \'eparse pour une fonction (g\'en\'erique) de l'espace de Hardy $H^p$, $0<p \leq 1$. Jusqu'ici les diff\'erentes d\'ecompositions atomiques obtenues ne permettaient pas d'avoir le caract\`ere \'epars des supports des atomes. Nous l'obtenons par une construction similaire \`a celle de la domination \'eparse, en utilisant un bon temps d'arr\^et bas\'e sur les fonctions carr\'ees localis\'ees, pour prendre en compte les oscillations d'une fonction de l'espace de Hardy.

Nous finissons par la sous-section \ref{sec:oscillation}, dans laquelle on d\'ecrit une sorte d'in\'egalit\'e de polarisation, du r\'esultat de Fefferman-Stein permettant d'estimer la norme $L^2$ d'une fonction par celle de sa fonction maximale di\`ese. Ce nouveau r\'esultat repose l\`a aussi sur une gestion optimis\'ee des oscillations, \`a travers un contr\^ole \'epars.

\section{Pr\'eliminaires}

\large{ \bf Notations.}
Lors de tout ce travail, on se placera dans ${\rr R}^n$ muni de sa structure Euclidienne, mais tout pourrait \^etre \'ecrit dans un espace homog\`ene. En effet, les techniques et outils utilis\'es reposent sur la structure dyadique et le caract\`ere doublant de l'espace ambiant.

\smallskip

Avant de faire quelques rappels sur les grilles dyadiques et les collections \'eparses, fixons quelques notations, qui seront r\'eguli\`erement utilis\'ees dans l'ensemble de l'article:
pour une boule (cube ou intervalle) de l'espace Euclidien, on note $\osc_B$ l'oscillation sur $B$ d\'efinie par
$$ \osc_B(f):= \big(\aver{B} \left| f- \aver{B} f dx \right| \, dx \big).
$$
Si $I$ est un intervalle, un cube ou une boule, on note aussi 
$\chi_{I}(x) := \left(1+ \frac{d(x,I)}{\ell(I)}\right)^{-1}$ la fonction de localisation autour de $I$. \\
Les diff\'erentes preuves pr\'esent\'ees ici auront pour but de construire une collection \'eparse de cubes dyadiques, celle-ci sera obtenue via des temps d'arr\^et. Tout au long du papier, $\ii I_{Stock}$ d\'enotera la collection des intervalles (ou cubes) disponibles au moment du temps d'arr\^et.
On notera aussi ${\mathcal M}$ la fonction maximale de Hardy-Littlewood, et ${\mathcal M}_p$ denote sa version $L^p$: $\ic M_p (f)= \left(  \ic M \vert f \vert^p \right)^{1/p}$.

\medskip

\large{ \bf Grilles dyadiques et collection \'eparses.}
Sur ${\rr R}^n$, associ\'ee \`a un param\`etre $\alpha\in\{0,\frac{1}{3}\}^n$, on peut consid\'erer la grille dyadique (translat\'ee)
$$ \Dy^\alpha := \left\{2^{-k}\left([0,1]^n+m+(-1)^k\alpha\right), \ k\in {\mathbb Z},\ m\in{\mathbb Z}^n \right\}$$
et on note $\Dy:= \cup_\alpha \Dy^\alpha$.
Pour un cube dyadique $Q^\alpha \in \Dy^\alpha$, on note $\ell(Q)$ la longueur de son cot\'e. Un r\'esultat bien connu est alors que pour tout cube $Q$ de ${\rr R}^n$, il existe $\alpha\in\{0,\frac{1}{3}\}^n$ et un cube $Q^\alpha\in \Dy^\alpha$ tel que
$$ Q\subset Q^\alpha \qquad \textrm{et} \qquad \ell(Q^\alpha)\leq 6 \ell(Q).$$

\begin{definition}\label{def:sparse} Une collection ${\mathcal S} \subset \Dy$ est dite $\eta$-{\it \'eparse} pour un certain $\eta\in(0,1)$ si on peut trouver une collection d'ensembles mesurables disjoints $(E_Q)_{Q\in {\mathcal S}}$ telle que pour tout $Q\in {\mathcal S}$
$$ E_Q \subset Q \qquad \textrm{et} \qquad |E_Q|\geq \eta |Q|.$$
\end{definition}

Il existe aussi la formulation suivante \'equivalente:

\begin{definition} \label{def:sparsebis} Une collection ${\mathcal S} \subset \Dy^0$ est dite {\'eparse}, si il existe $\eta \in (0,1)$ tel que pour tout $Q \in \ic S$, 
\[
\sum_{P \in ch_{\ic S}\left( Q\right)} |P| \leq \eta |Q|,
\]
o\`u $ch_{\ic S}(Q)$ repr\'esente la collection des descendants directs de $Q$ dans $\ic S$, c'est \`a dire les intervalles maximaux de $\ic S$ strictement contenus dans $Q$.
\end{definition}

\begin{definition}\label{def:carleson} Une collection ${\mathcal S} \subset \Dy$ est dite $\Lambda$-{\it Carleson} pour un certain $\Lambda> 1$ si pour tout $Q\in {\mathcal S}$ on a
$$ \sum_{\genfrac{}{}{0pt}{}{P\in{\mathcal S}}{P\subset Q}} |P| \leq \Lambda |Q|.$$
\end{definition}

On a alors l'\'equivalence entre ces deux notions (\cite[Lemme 6.3]{LernerNazarov}):

\begin{proposition} Une collection ${\mathcal S} \subset \Dy$ est $\eta$-\'eparse si et seulement si elle est $\eta^{-1}$-Carleson.
\end{proposition}

Cette \'equivalence nous permet donc d'utiliser simultan\'ement les propri\'et\'es 'faciles' dues \`a chacune des propri\'et\'es (voir \cite{LernerNazarov}). Dans la suite, nous appellerons une collection \'eparse, si elle est $\eta$-\'eparse avec un param\`etre $\eta\in (0,\eta_0)$ pour un certain param\`etre $\eta_0<1$ implicite.

Nous avons aussi la propri\'et\'e suivante (\cite[Sections 12-13]{LernerNazarov}):

\begin{proposition} \label{prop:enlarge_sparse} Soit ${\mathcal S}$ une collection \'eparse et $\rho$ une fonction continue croissante sur $[0,1]$ telle que $\rho(0)=0$. Supposons que $\rho$ v\'erifie
$$ \sum_{k\geq 1} k \rho(2^{-k}) <\infty.$$ 
Alors pour toute fonction $f,g\in L^1_{\loc}$, il existe une collection \'eparse enlargie $\widetilde \mcS$ telle que
$$ \sum_{k\geq 0} \rho(2^{-k}) \sum_{Q\in \mcS} \left(\aver{2^k Q} |f| dx \right)\left(\aver{2^k Q} |g| dx \right) |Q| \lesssim  \sum_{Q\in \widetilde \mcS} \left(\aver{Q} |f| dx \right)\left(\aver{Q} |g| dx \right) |Q|.$$
Les constantes implicites sont ind\'ependantes de $f,g$ et de $\mcS$ et ne d\'ependent que de $\rho$.
\end{proposition}

\section{La propri\'et\'e du type faible $(1,1)$ et de la bornitude $L^\infty$-$\BMO$}

\subsection{Type faible $(1,1)$ comme cons\'equence d'un contr\^ole \'epars}

\begin{proposition} \label{prop:typefaible} Soit $T$ un op\'erateur lin\'eaire born\'e sur $L^2$, v\'erifiant un contr\^ole \'epars via la forme bilin\'eaire. Plus pr\'ecis\'ement, supposons que pour toutes fonctions (compactement support\'ees) $f,g\in L^2$ il existe une collection \'eparse $\mcS$ telle que
$$ \left| \langle Tf,g \rangle \right| \lesssim \sum_{Q\in \mcS} \left(\aver{Q} |f|\, dx \right) \left(\aver{Q} |g|\, dx \right) |Q|.$$
Alors $T$ est de type faible $(1,1)$.
\end{proposition}

Ce r\'esultat n'est pas nouveau, et on peut trouver une preuve dans \cite[Th\'eor\`eme E]{CCDO}. Nous allons cependant donner une preuve un peu diff\'erente, qui met en lumi\`ere la connexion avec la d\'ecomposition de Calder\'on-Zygmund.

\begin{proof} Nous allons utiliser la caract\'erisation suivante de l'espace $L^{1,\infty}$: pour une fonction mesurable $\phi$ alors $\phi \in L^{1,\infty}$ si
\begin{equation} \|\phi\|_{L^{1,\infty}} \simeq \sup_{E \subset {\rr R}^n} \inf_{\genfrac{}{}{0pt}{}{E \subset E'}{2 |E'|\geq |E|}}  \int_{E'} |f(x)|\, dx. \label{eq:caracterisationL1faible} \end{equation}
On fixe donc une fonction $f\in L^1\cap L^2$, normalis\'ee dans $L^1$ ($\|f\|_1=1$) et on souhaite v\'erifier que $T(f) \in L^{1,\infty}$. Donc pour tout ensemble mesurable $E$ de mesure finie, on consid\`ere l'ensemble $E'$ d\'efini par
$$ E':=\left\{ x\in E, \ M(f)(x)<K|E|^{-1} \right\},$$
pour une constante $K$ num\'erique suffisamment grande. Cet ensemble est majeur dans $E$ et v\'erifie $2|E'|\geq |E|$. Il nous reste alors \`a estimer
\begin{equation} \int_{E'} |Tf(x)|\, dx = \sup_{h} \left| \langle Tf, h \rangle \right| \label{eq:sup}
\end{equation}
o\`u le supremum est pris sur toutes les fonctions $h\in L^2$ support\'ee sur $E'$ avec $\|h\|_\infty\leq 1$.
Fixons une telle fonction $h$ et utilisons l'hypoth\`ese de contr\^ole \'epars. Il existe donc une collection \'eparse $\mcS$ telle que
$$ \left| \langle Tf, h \rangle \right| \lesssim  \sum_{Q\in \mcS} \left(\aver{Q} |f|\, dx \right) \left(\aver{Q} |h|\, dx \right) |Q|.$$
Supposons que toute la collection \'eparse soit incluse dans une des grilles dyadiques: il existe $\alpha$ tel que $\mcS \subset \Dy^\alpha$ (ce qu'il est toujours possible de supposer). 
Utilisons maintenant une d\'ecomposition de Calder\'on-Zygmund de $|f|\in L^1$ au niveau $K|E|^{-1}$: il existe une 'bonne' fonction $g$, des cubes $Q_i$, et des 'mauvaises' fonctions $b_i$ telles que 
$$ |f| = g + \sum_i b_i, \quad \|g\|_\infty \lesssim |E|^{-1}, \quad \|g\|_1\lesssim 1$$
et $b_i$ est support\'e dans $Q_i$ avec une int\'egrale nulle. En fait, les cubes $Q_i$ peuvent \^etre choisis comme un recouvrement maximal dyadique de l'ensemble de niveau $\{x,\ M(f)(x)> K|E|^{-1} \}$ de sorte que tous les cubes $Q_i$ sont inclus dans $(E')^c$. 
En utilisant la bornitude $L^2$ des formes \'eparses, on obtient que
\begin{align*}
 \sum_{Q\in \mcS} \left(\aver{Q} g\, dx \right) \left(\aver{Q} |h|\, dx \right) |Q| & \lesssim \int {\mathcal M}(|g|) {\mathcal M}(|h|) \, dx \\
 & \lesssim \|g\|_2 \|h\|_2 \lesssim \left(\|g\|_\infty \|g\|_1\right)^{1/2} \|h\|_\infty |E'|^{1/2} \\
 & \lesssim 1.
\end{align*}
Pour chaque cube $Q_i$ fix\'e, puisque $b_i$ est d'int\'egrale nulle et est support\'ee sur $Q_i$, la structure dyadique entra\^ine:
$$  \sum_{Q\in \mcS} \left(\aver{Q} b_i\, dx \right) \left(\aver{Q} |h|\, dx \right) |Q| =  \sum_{\genfrac{}{}{0pt}{}{Q\in \mcS}{Q\subset Q_i}} \left(\aver{Q} g\, dx \right) \left(\aver{Q} |h|\, dx \right) |Q| = 0$$
car $h$ est suppos\'ee \^etre support\'ee sur $E'$ qui ne rencontre pas $Q_i$.
\`A la fin,  on obtient que
$$ \left| \langle Tf, h \rangle \right| \lesssim 1$$
uniform\'ement en $h$, ce qui entra\^ine (par \eqref{eq:sup}) 
$$ \int_{E'} |Tf(x)|\, dx \lesssim 1$$ et donc conclut la preuve de $\|Tf\|_{L^{1,\infty}} \lesssim 1$.
\end{proof}

\begin{remark}
\begin{itemize}
\item[$\bullet$] L'hypoth\`ese de la bornitude $L^2$ de l'op\'erateur n'est pas importante, c'est juste pour donner un cadre et pouvoir \'evaluer $T(f)$, pour $f\in L^2$.

\item[$\bullet$] On observe donc que l'\'evaluation du type faible $L^1$ de l'op\'erateur $T$, correspond exactement \`a l'\'evaluation de l'op\'erateur sur la 'bonne' partie de la d\'ecomposition de Calder\'on-Zygmund.
\end{itemize}
\end{remark}

Le contr\^ole d'un op\'erateur lin\'eaire, par une forme bilin\'eaire \'eparse permet d'obtenir un contr\^ole par une forme auto-adjointe (alors que l'op\'erateur consid\'er\'e initialement ne l'est pas forc\'ement), qui implique le type faible $L^1$. On en d\'eduit donc le corollaire suivant:

\begin{corollary} \label{coro:typefaible} 
Soit $T$ un op\'erateur lin\'eaire, tel que $T$ admet un contr\^ole par une forme bilin\'eaire \'eparse. Alors l'op\'erateur dual $T^*$ est de type faible $L^1$.
\end{corollary}

\subsection{Lien entre la propri\'et\'e \'eparse d'une collection et $BMO$} 


\begin{proposition} Soit $(I)_I$ une collection d'intervalles dyadiques et $\phi:= \sum_{I} \tilde h_I$, o\`u $\tilde h_I$ sont les fonctions de Haar $L^\infty$-normalis\'ees. Alors la collection est \'eparse si et seulement si $\phi \in BMO$
\end{proposition}

\begin{proof}
Cette proposition est une cons\'equence de la caract\'erisation des fonctions de $BMO$ par des mesures de Carleson et fonctions carr\'ees (\cite[Th\'eor\`eme 3 page 159]{SteinBigBook}). Ou, encore plus direct, on peut utiliser la caract\'erisation de $BMO$ avec des ondelettes, comme d\'ecrite dans \cite{Yves-ondelettes}: ``les coefficients d'ondelettes d'une fonction de $BMO$ satisfont les c\'el\`ebres conditions quadratiques de Carleson''. 

En effet, si on emploie les ondelettes de Haar, la condition \'eparse, qui traduit une condition d'ordre $1$ de Carleson (definition \ref{def:carleson}), devient \'equivalente \`a l'appartenance de certaines fonctions $\phi$ \`a $BMO$. Pour une telle fonction $\phi$, on doit avoir
\[
\vert \langle \phi, h_I \rangle \vert ^2 = \vert I \vert, \quad \text{ pourt tout $I$ dans la collection \'eparse}.
\]
On peut donc prendre $\phi=\sum_{I \in \ic S} \epsilon_I \tilde h_I$, pour toute suite $(\epsilon_I)_I = \pm 1$.

\end{proof}



\section{La propri\'et\'e ``$T(1)=0$'' et application pour la composition d'op\'erateurs} \label{sec:T1=0}

On rappelle qu'un op\'erateur lin\'eaire born\'e sur $L^2$ est dit de Calder\'on-Zygmund si il est associ\'e \`a un noyau $K:{\rr R}^n \times {\rr R}^n \rightarrow {\rr R}$ tel que pour un certain $\epsilon>0$ et tout $x\neq y \in{\rr R}^n$ alors
$$ |K(x,y)| \lesssim \frac{1}{|x-y|^{n}}$$
et pour tout $z\in {\rr R}^n$ avec $2|z|<|x-y|$
$$  |K(x+z,y)-K(x,y)|  + |K(x,y+z)-K(x,y)| \lesssim \frac{|z|^\epsilon}{|x-y|^{n+\epsilon}}.$$
Il est maintenant classique que si $T$ est un op\'erateur lin\'eaire de Calder\'on-Zygmund alors on peut d\'efinir $T(1)$ dans BMO (et plus g\'en\'eralement montrer que $T$ admet une extension continue de $L^\infty$ dans BMO). L'espace BMO \'etant d\'efini modulo les constantes, la condition $T(1)=0\in BMO$ signifie plus exactement que $T(1)$ est une fonction constante.

\begin{proposition} \label{prop:question1} Soit $T$ un op\'erateur de Calder\'on-Zygmund v\'erifiant $T(1)=0 \in BMO$. Alors pour toute fonction $f\in L^2$ il existe une collection \'eparse $\mcS$ telle que, pour presque tout $x\in {\rr R}^n$ 
$$ \left|T(f)(x)\right| \lesssim \sum_{Q\in \mcS} \osc_Q(f) {\bf 1}_Q(x).$$
\end{proposition}

L'oscillation $\osc_Q(f)$ est plus pr\'ecise que la moyenne et encode exactement le caract\`ere $T(1)=0$. On obtient donc un contr\^ole \'epars, en pr\'eservant cette propri\'et\'e \`a l'\'echelle \'el\'ementaire de la repr\'esentation.

\begin{proof} Il n'est pas clair comment utiliser la propri\'et\'e $T(1)=0$ \`a travers les plus r\'ecentes preuves de Lacey \cite{LaceyA_2} ou Lerner \cite{Lerner4}, qui utilisent un 'bon' op\'erateur maximal tronqu\'e. Nous pr\'ef\'erons donc utiliser la preuve initiale de Lerner \cite{Lerner-simplerA_2} qui repose sur les oscillations locales en moyenne. Rappelons tout d'abord sa d\'efinition: pour $\phi$ une fonction mesurable sur ${\rr R}^n$ et une boule $B$, l'oscillation locale en moyenne de $\phi$ sur $B$ au niveau $\lambda\in(0,1)$ est d\'efinie par
$$ \omega_\lambda(\phi;B):= \inf_{c\in {\rr R}} \left((\phi-c)1_B\right)^*(\lambda|B|)$$
o\`u $^*$ correspond au r\'earrangement d\'ecroissant. Fixons alors le niveau $\lambda$ suffisamment petit, en utilisant la formule de Lerner (see \cite{Lerner1}) et \cite[Theorem 10.2]{LernerNazarov}, pour toute fonction $f\in L^2$ (de sorte que $Tf$ a aussi un sens dans $L^2$) il existe une collection \'eparse $\mcS=(Q)_Q$ de cubes telle que pour presque tout $x$, nous avons
$$ \left|Tf(x)\right| \lesssim \sum_{Q\in \mcS} \omega_{\lambda}(Tf;Q) {\bf 1}_Q(x).$$
De plus, $T$ \'etant un op\'erateur de Calder\'on-Zygmund, on sait que pour un certain $\delta:=\delta(n,\epsilon)$
$$ \omega_{\lambda}(Tf;Q) \lesssim \sum_{\ell \geq 0} 2^{-\ell \delta} \left(\aver{2^\ell Q} |f|\, d\mu\right).$$
\'Etant donn\'e que $T(1)=0\in BMO$ et que l'oscillation est invariante par les constantes, on peut soustraire la moyenne et obtenir
$$ \omega_{\lambda}(Tf;Q) \lesssim \sum_{\ell \geq 0} 2^{-\ell \delta} \left(\aver{2^\ell Q} \left|f - \aver{Q}f d\mu \right|\, d\mu\right).$$
Par cons\'equent, il vient
\begin{align*}
\left|Tf(x)\right| & \lesssim \sum_{\ell \geq 0} \sum_{Q\in \mcS}  2^{-\ell \delta} \left(\aver{2^\ell Q} \left|f - \aver{Q}f d\mu \right|\, d\mu\right) {\bf 1}_Q(x).
\end{align*}
En utilisant que 
$$ \left(\aver{2^\ell Q} \left|f - \aver{Q}f d\mu \right|\, d\mu\right) \lesssim \sum_{k=0}^\ell \osc_{2^k Q}(f)$$
on en d\'eduit
\begin{align*}
\left|Tf(x)\right| & \lesssim \sum_{\ell \geq 0} \sum_{Q\in \mcS}  \sum_{k=0}^\ell 2^{-\ell \delta} \osc_{2^k Q}(f)  {\bf 1}_Q(x) \lesssim \sum_{k\geq 0} \sum_{Q\in \mcS}  2^{-k \delta} \osc_{2^k Q}(f) {\bf 1}_Q(x).
\end{align*}
Il nous suffit maintenant d'estimer la double somme par une somme sur une seule collection \'eparse. Cette \'etape est expliqu\'ee dans \cite[Sections 12-13]{LernerNazarov}, en utilisant un temps d'arr\^et suppl\'ementaire. Il repose uniquement sur le fait que pour une collection disjointe de boules $B_1,...,B_N$ incluses dans $B$ alors
\begin{align*}  
\sum_{i=1}^N |B_i| \osc_{B_i}(f) & \leq  2 \sum_{i=1}^N |B_i| \left(\aver{B_i} \left|f - \aver{B}f d\mu \right|\, d\mu\right) \\
& \leq 2 \osc_{B}(f).
\end{align*}
Cette propri\'et\'e est utilis\'ee pour v\'erifier que la collection agrandie est bien toujours \'eparse. Donc en appliquant cet argument, il existe une collection \'eparse agrandie
$\widetilde \mcS$ telle que
\begin{align*}
\left|Tf(x)\right| \lesssim  \sum_{Q\in \widetilde \mcS}  \osc_{Q}(f) {\bf 1}_Q(x).
\end{align*}
\end{proof}

Nous allons maintenant \'etudier la composition de deux op\'erateurs de Calder\'on-Zygmund, $S$ et $T$. Par composition des estimations \`a poids, nous savons que pour tout $p\in(1,\infty)$ et tout poids $\omega \in A_p$ alors
$$ \|T\circ S \|_{L^p_\omega \to L^p_\omega} \lesssim [\omega]_{A_p}^{2\max\{1, (p-1)^{-1}\}}.$$
Nous allons voir comment ces estimations peuvent \^etre am\'elior\'ees en fonction d'une condition du type $T(1)$ et/ou $S^* (1)=0$. Tout d'abord, nous observons que pour am\'eliorer cette estimation, nous avons besoin de comprendre profond\'ement la composition en utilisant que les deux op\'erateurs $T$ et $S$ ont des ``annulations'' qui doivent int\'eragir. Plus pr\'ecis\'ement, pour la composition $T \circ S$, les interactions entre $T$ et $S^*$ semblent \^etre les plus importantes et il est donc naturel de s'int\'eresser aux conditions $T(1)=0$ et $S^* (1)=0$. Ceci a d\'ej\`a \'et\'e observ\'e dans \cite[Section 9]{CoifMeyer-ondelettes} pour montrer que la composition $T\circ S$ est encore un op\'erateur de Calder\'on-Zygmund sous la condition $T(1)=0=S^*(1)$.

\begin{itemize}
 \item[$\bullet$] Si $T,S$ sont deux op\'erateurs de Calder\'on-Zygmund, alors ils sont tous les deux control\'es ponctuellement par un op\'erateur \'epars. Il est alors facile d'en d\'eduire que pour tout exposant $r\in(1,\infty)$ nous avons le contr\^ole suivant: pour toutes fonctions $f,g\in L^2$ il existe une collection \'eparse $\mcS$ telle que
 $$ \left|\langle T\circ S f, g \rangle \right|= \left|\langle  S f, T^*g \rangle \right| \lesssim \sum_{Q\in \mcS} \left(\aver{Q} |f|^r\, dx \right)^{1/r} \left(\aver{Q} |g|^{r}\, dx \right)^{1/r} |Q|.$$
 En utilisant \cite[Proposition 6.4]{BernicotFreyPetermichl}, on obtient que pour tout $p\in (1,\infty)$, tout exposant $r>1$ tel que $r<p<r'$ et tout poids $\omega \in A_{p/r} \cap RH_{(r'/p)'}$ alors
 $$  \|T\circ S \|_{L^p_\omega \to L^p_\omega} \lesssim \left([\omega]_{A_{p/r}} [\omega]_{RH_{(r'/p)'}}\right)^{\alpha},$$
 avec $\alpha:=\max\{\frac{1}{p-r}, \frac{r'-1}{r'-p}\}$.
On peut donc am\'eliorer l'exposant sur la caract\'eristique du poids, si on s'autorise \`a perdre un peu sur la classe du poids.

\item[$\bullet$] Dans l'autre cas extr\^eme o\`u on suppose une double condition $T(1)=0$ et $S^*(1)=0$ alors nous savons de \cite[Section 9]{CoifMeyer-ondelettes} qu'il existe suffisamment d'interaction de sorte que $T\circ S$ est encore un op\'erateur de Calder\'on-Zygmund et est donc born\'e ponctuellement par un op\'erateur \'epars avec des moyennes $L^1$.

 \item[$\bullet$] Inter\'essons nous maintenant au cas interm\'ediaire o\`u nous supposons seulement que $T$ v\'erifie la condition $T(1)=0$. Alors en utilisant la proposition pr\'ec\'edente avec le fait que pour tout cube $Q$ et tout exposant $r\in(1,\infty)$ 
 $$ \osc_Q(Sf) \lesssim \left(\aver{Q} |f|^{r}\right)^{1/r} + \inf_{Q} M[f]$$
on peut d\'eduire que $T\circ S$ admet un contr\^ole ponctuel par un op\'erateur \'epars avec des moyennes $L^r$. Cette observation permet ainsi d'am\'eliorer le contr\^ole de la forme bilin\'eaire de la fa\c{c}on suivante:
 $$ \left|\langle T\circ S f, g \rangle \right|= \left|\langle  S f, T^*g \rangle \right| \lesssim \sum_{Q\in \mcS} \left(\aver{Q} |f|^r\, dx \right)^{1/r} \left(\aver{Q} |g|\, dx \right) |Q|,$$
 o\`u l'exposant d'int\'egrabilit\'e des moyennes de $g$ est maintenant \'egal \`a $1$. Selon \cite{BernicotFreyPetermichl} pour les estimations \`a poids obtenues par une telle forme bilin\'eaire, nous obtenons que l'on peut supprimer la condition de H\"older inverse $RH$, dans la classe de poids consid\'er\'ee: pour tout exposant $r\in(1,p)$ et tout poids $\omega \in A_{p/r}$ alors
 $$  \|T\circ S \|_{L^p_\omega \to L^p_\omega} \lesssim \left([\omega]_{A_{p/r}}\right)^{\alpha},$$
 avec $\alpha:=\max\{\frac{1}{p-r}, 1\}$.
Par sym\'etrie, si $S^*(1)=0$ alors l'exposant d'int\'egrabilit\'e des moyennes de $f$ peut \^etre choisi \'egal \`a 1.
\end{itemize}

\section{Le contr\^ole par le gradient et estimations optimales dans des espaces de Sobolev \`a poids} \label{sec:gradient}

\begin{proposition} 
Soit $T$ un op\'erateur de Calder\'on-Zygmund avec un noyau $K$ v\'erifiant $|\nabla_{x,y} K(x,y)| \lesssim |x-y|^{-n-1}$ et $|\nabla^2_{x,y} K(x,y)| \lesssim |x-y|^{-n-2}$. Supposons de plus que
\begin{itemize}
\item[$\bullet$] $T(1)=0$ dans BMO
\item[$\bullet$] L'op\'erateur $\nabla T$ v\'erifie l'estimation $L^2$ suivante:
$$  \| \nabla T(f)\|_2 \lesssim \|\nabla f\|_2, \qquad  \forall f\in W^{1,2}.$$
\end{itemize}
Alors pour toute fonction $f\in W^{1,2}$, il existe une collection \'eparse $\mcS=(Q)_{Q\in \mcS}$ telle que
$$ |\nabla Tf(x)| \lesssim \sum_{Q} \left(\aver{Q} \left|\nabla f\right| \right) {\bf 1}_Q(x).$$
\end{proposition}

\begin{remark} On observe que:
\begin{itemize}
\item[$\bullet$] Une r\'egularit\'e d'ordre $1+\epsilon$ sur le noyau est en fait suffisante.
\item[$\bullet$] Cette propri\'et\'e traduit une certaine commutativit\'e entre la domination \'eparse et le gradient dans le cas des op\'erateurs de Calder\'on-Zygmund suffisamment r\'eguliers.
\item[$\bullet$] La preuve que nous allons d\'etailler reprend la preuve de Lerner \cite{Lerner-simplerA_2}, bas\'ee sur une repr\'esentation avec des oscillations locales. Nous la d\'etaillons, pour voir comment y int\'egrer le gradient.
\end{itemize}
\end{remark}

\begin{proof}
{\bf \'Etape 1:} Nous allons d'abord v\'erifier que l'op\'erateur $\nabla T$ v\'erifie une estimation de type faible $L^1$, plus pr\'ecis\'ement:
\begin{equation} \| \nabla T f \|_{L^{1,\infty}} \lesssim \|\nabla f\|_{L^1}. \label{eq:etape1}
 \end{equation}
Pour cela, nous allons utiliser une d\'ecomposition de Calder\'on-Zygmund pour les fonctions ``gradient'', d\'evelopp\'ee par Auscher dans \cite{Auscher-memoire} (voir aussi \cite{Auscherannexe} pour une explication suppl\'ementaire). La preuve n'est pas originale et a \'et\'e souvent utilis\'ee dans le cadre de diff\'erentes transform\'ees de Riesz. Pour \^etre complet, nous d\'etaillons la preuve dans ce contexte. Fixons donc une fonction $f\in W^{1,2}$ telle que $|\nabla f|\in L^1$ ainsi qu'un niveau $\alpha>0$. Nous savons qu'il existe une ``bonne'' fonction $g$ et une collection de cubes $Q_i$ v\'erifiant la propri\'et\'e de recouvrement born\'e ainsi que des fonctions $b_i\in W^{1,1}_0(Q_i)$ support\'ees dans $Q_i$ telles que
$$ \|\nabla g \|_{\infty} \lesssim \alpha, \qquad \|\nabla g\|_1 \lesssim \|\nabla f \|_1,$$
$$ \sum_i |Q_i| \lesssim \alpha^{-1} \|\nabla f\|_1   \qquad \textrm{et} \qquad \aver{Q_i} |\nabla b_i|\, dz \lesssim \alpha.$$
On souhaite maintenant estimer 
$$ \left| \{x,\ |\nabla Tf|>\alpha\} \right| \leq \left| \{x,\ |\nabla Tg|>\alpha/2\} \right| + \left| \{x,\ |\nabla Tb|>\alpha/2\} \right|,$$
avec $b=\sum b_i$. Pour le premier ensemble de niveau, on utilise l'hypoth\`ese $L^2$ pour d\'eduire que
$$ \left| \{x,\ |\nabla Tg|>\alpha/2\} \right| \lesssim \alpha^{-2} \|\nabla Tg\|_2^2 \lesssim \alpha^{-2} \|\nabla g\|_2^2 \lesssim \alpha^{-1} \|\nabla f\|_1.$$
Pour le second ensemble de niveau, \'etant donn\'e que $\cup_i (2Q_i)$ a une mesure controll\'ee par $\alpha^{-1} \|\nabla f\|_1 $, nous avons seulement \`a estimer
$$ \left| \{x \in \cap_i (2Q_i)^c,\  |\nabla Tb (x)|>\alpha/2\} \right|.$$
Pour chaque indice $i$ et un point $x\in (2Q_i)^c$, on note $c_i$ (resp. $\ell_i$) le centre (resp. la longueur du cot\'e) de $Q_i$ et alors par la condition $T(1)=0\in BMO$, on a
\begin{align*} 
\nabla T(b_i)(x) & = \nabla T \left(b_i - \aver{Q_i} b_i \right)(x) \\
 & = \nabla T \left[\left(b_i - \aver{Q_i} b_i \right)\chi_i \right](x) + \left(\aver{Q_i} b_i\right) \nabla T(1-\chi_i)(x) 
\end{align*}
o\`u $\chi_i$ est une fonction Lipschitz, \'egale \`a $1$ sur $Q_i$, support\'ee sur $2Q_i$ et avec un gradient uniform\'ement born\'e par $\ell_i^{-1}$.
Par l'estimation de r\'egularit\'e sur le noyau (puisque $d(x,Q_i)\simeq |x-c_i|$), on obtient
\begin{align*}
\left| \nabla T \left(b_i - \aver{Q_i} b_i \right)(x) \right| & \lesssim \int_{Q_i} \left|b_i(y) - \aver{Q_i} b_i \right| \frac{1}{|x-c_i|^{n+1}} dy \\
& \lesssim \left(\frac{\ell_i}{|x-c_i|}\right)^{n+1} \ell_i^{-1}\osc_{Q_i}(b_i) \\
& \lesssim \left(\frac{\ell_i}{|x-c_i|}\right)^{n+1} \left(\aver{Q_i} |\nabla b_i| \, dz \right) \lesssim \left(\frac{\ell_i}{|x-c_i|}\right)^{n+1} \alpha,
\end{align*}
o\`u nous avons utilis\'e que $b_i$ est support\'e sur $Q_i$ et l'in\'egalit\'e de Poincar\'e $L^1$. De mani\`ere similaire, nous obtenons
\begin{align*}
\left| \left(\aver{Q_i} b_i\right) \nabla T(1-\chi_i)(x) \right| & \lesssim \ell_i \left(\aver{Q_i} |\nabla b_i| \, dz \right) \left|\nabla T(\chi_i)(x)\right| \\
& \lesssim \ell_i \left(\aver{Q_i} |\nabla b_i| \, dz \right) \frac{\ell_i^n}{|x-r_i|^{n+1}} \\
& \lesssim \left(\frac{\ell_i}{|x-c_i|}\right)^{n+1} \alpha.
\end{align*}
En cons\'equence, pour tout $x \in \cap_i (2Q_i)^c$ nous avons montr\'e que
 $$ |\nabla Tb(x)| \lesssim \alpha \sum_i  \left(\frac{\ell_i}{\ell_i + |x-c_i|}\right)^{n+1}, $$
 ce qui est exactement la fonction de Marcinkiewicz associ\'ee \`a la collection $(Q_i)$, qui v\'erifie une estimation de type faible $L^1$ puisque la collection v\'erifie la propri\'et\'e de recouvrement born\'e. On conclut donc que
$$\left| \left\{x,  \sum_i \sum_i  \left(\frac{\ell_i}{\ell_i + |x-c_i|}\right)^{n+1} \gtrsim 1\right\}\right| \lesssim \sum_i |Q_i| \lesssim \alpha^{-1}\|\nabla f\|_1,$$
 ce qui finit la preuve de $ \left| \{x,\ |\nabla Tf|>\alpha\} \right|\lesssim \alpha^{-1} \|\nabla f\|_1.$
 
 \smallskip
 {\bf \'Etape 2:} Nous allons maintenant obtenir le contr\^ole \'epars de l'op\'erateur $\nabla T$. En utilisant la formule de Lerner (sur les oscillations locales) (on renvoie le lecteur \`a la preuve de la Proposition \ref{prop:question1} pour les notations), \cite{Lerner1} et \cite[Theorem 10.2]{LernerNazarov}: pour toute fonction $f\in W^{1,2}$ (de sorte que $\nabla Tf$ a aussi un sens dans $L^2$) il existe une collection \'eparse $\mcS=(Q)_Q$ de cubes telle que pour presque tout $x$, on a avec un certain param\`etre $\lambda$
$$ \left|\nabla Tf(x)\right| \lesssim \sum_{Q\in \mcS} \omega_{\lambda}(\nabla Tf;Q) {\bf 1}_Q(x).$$
Fixons une des boules $Q$, et une fonction Lipschitz $\chi_Q$, \'egale \`a $1$ sur $2Q$, support\'ee sur $4Q$ et avec un gradient uniform\'ement born\'e par $\ell(Q)^{-1}$. Alors puisque $T(1)=0\in BMO$, on a $ \nabla T(f) = \nabla T(f_1)+ \nabla T(f_2)$ avec
$$ f_1:=\left(f- \aver{Q} f \right) \chi_Q \qquad \textrm{et} \qquad  f_2:=\left(f- \aver{Q} f \right) \left(1-\chi_Q\right).$$
En utilisant l'in\'egalit\'e \eqref{eq:etape1} sur la premi\`ere fonction $f_1$ avec l'in\'egalit\'e de Poincar\'e $L^1$, on obtient que
\begin{align*}
\omega_{\lambda}(\nabla Tf_1;Q) & \lesssim \frac{1}{|Q|} \left\| \left(f- \aver{Q} f \right) \chi_Q \right\|_{L^1} \\
 & \lesssim \frac{1}{\ell(Q)} \osc_{4Q}(f) + \left(\aver{4Q} |\nabla f| \, dz \right) \\
 & \lesssim \left(\aver{4Q} |\nabla f| \, dz \right).
 \end{align*}
Pour la seconde partie, on utilise $c$ le centre du cube $Q$, pour avoir
\begin{align*}
\omega_\lambda(\nabla Tf_2;Q) & \lesssim \| \nabla Tf_2 - \nabla Tf_2(c)\|_{L^\infty(Q)}.
\end{align*}
Pour tout point $x\in Q$, on v\'erifie en utilisant la r\'egularit\'e du noyau $K$ que
\begin{align*}
\left|\nabla Tf_2(x) - \nabla Tf_2(c)\right| & \lesssim \int_{y\in (2Q)^c} \left|\nabla K(x,y)-\nabla K(c,y)\right| |f(y)- f_Q| \, dy \\
& \lesssim \sum_{j\geq 0} 2^{-j} \frac{1}{2^j \ell(Q)} \left(\aver{2^j Q} \left|f-f_Q\right| \, dy \right) \\
& \lesssim \sum_{j\geq 0} \sum_{k=0}^j 2^{-j} \left( \aver{2^k Q} |\nabla f| \, dy \right) \\
& \lesssim \sum_{k\geq 0} 2^{-k} \left( \aver{2^k Q} |\nabla f| \, dy \right),
\end{align*}
o\`u nous avons utilis\'e l'in\'egalit\'e de Poincar\'e $L^1$ et la notation $f_Q$ pour la moyenne de $f$ sur la boule $Q$.
On conclut donc que 
$$ \omega_{\lambda}(\nabla Tf;Q) \lesssim \sum_{k\geq 0} 2^{-k} \left( \aver{2^k Q} |\nabla f| \, dy \right),$$
d'o\`u
$$ \left|\nabla Tf(x)\right| \lesssim \sum_{k\geq 0} 2^{-k} \sum_{Q\in \mcS}  \left( \aver{2^k Q} |\nabla f| \, dy \right) {\bf 1}_Q(x).$$
On obtient le r\'esultat \'enonc\'e en utilisant la m\'ethode de \cite[Sections 12-13]{LernerNazarov}, qui permet d'estimer la double somme par une seule somme sur une collection \'eparse.
\end{proof}

\begin{corollary} Soit $T$ un op\'erateur comme dans la proposition pr\'ec\'edente. Alors l'op\'erateur $T$ v\'erifie des estimations optimales dans des espaces de Sobolev \`a poids: pour tout $p\in (1,\infty)$ et tout poids $\omega\in A_p$ alors $T$ est continu sur $W^{1,p}(\omega)$ et
$$ \| T\|_{W^{1,p}(\omega) \to W^{1,p}(\omega)} \lesssim [\omega]_{A_p}^{\max\{1,(p-1)^{-1}\}}.$$

\end{corollary}

\section{Repr\'esentation de multiplicateurs de Fourier g\'en\'eralis\'es, applications \`a la bornitude dans des espaces de Hardy \`a poids}  \label{sec:haar}

Dans cette section, nous allons nous int\'eresser au contr\^ole \'epars d'un op\'erateur, admettant une certaine repr\'esentation (g\'en\'eralisant le cas des multiplicateurs de Fourier).

Pour simplifier, nous allons travailler en dimension 1 avec les intervalles dyadiques $\Dy^0$ de ${\mathbb R}$. Tous les r\'esultats s'\'etendent de fa\c{c}on automatique au cas de la dimension plus grande avec les cubes dyadiques.

Plus pr\'ecis\'ement, on consid\`ere un op\'erateur lin\'eaire $T$ born\'e sur $L^2$ tel qu'il existe des coefficients $a_I(f)$, index\'es par les intervalles dyadiques $I \in{\mathbb D}^0$, et tel que, pour toutes fonctions $f,g\in L^2$
$$ |\langle Tf,g\rangle| \lesssim \sum_{I\in {\mathbb D}^0} \vert a_I(f)\vert \cdot \vert a_I(g)\vert.$$

\begin{remark} On pourrait supposer une d\'ecomposition sur l'ensemble des intervalles dyadiques ${\mathbb D}$. Une telle repr\'esentation g\'en\'eralise la d\'ecomposition en ondelettes d'un multiplicateur de Fourier, et permet de consid\'erer d'autres types d'op\'erateurs (voir section \ref{sec:leray}). \\
De plus on pourrait consid\'erer deux types de coefficients $a_I(f)$ et $b_I(g)$, en supposant les m\^emes hypoth\`eses sur les deux collections de coefficients.
\end{remark}

Les coefficients $a_I(f)$ doivent \^etre consid\'er\'es comme une version \'el\'ementaire localis\'ee de l'op\'erateur $T$ en espace (autour de l'intervalle $I$) et en fr\'equence (\`a l'\'echelle $\ell(I)^{-1}$). L'exemple canonique dans cette situation est donn\'e par $a_I(f)=\langle f, \phi_I \rangle$, o\`u $\phi_I$ est un paquet d'onde: $\phi_I$ d\'ecroit tr\`es vite loin de $I$ et $\hat{\phi}_I$ est support\'e dans $\big[ 1/\ell(I), 2/\ell(I)  \big]$.

En suivant la terminologie de \cite{MTTBiest0} ou la pr\'esentation de \cite{multilinear_harmonic}: pour un intervalle $I_0$, on d\'efinit le '{\it size} maximal' par
$$ \size_{I_0}f := \sup_{I\subset I_0} \Big(\frac{1}{|I|} \sum_{J\subset I} |a_J(f)|^2\Big)^{1/2}.$$
Nous allons faire les deux hypoth\`eses suivantes: 
\begin{itemize}
 \item[$\bullet$] Condition d' \'energie: Pour tout intervalle $I_0$
 \begin{equation} \sum_{I\subset I_0} |a_I(f)|^2 \lesssim  \|f \cdot\chi_{I_0}^M\|_{2}^2 \label{eq:energie} \end{equation}
 o\`u $M$ est un entier assez grand.
 \item[$\bullet$] Condition de Carleson sur le ``size'': pour tout intervalle $I_0$
 \begin{equation}
  \size_{I_0}f \ \lesssim \ \sup_{I\subset I_0} \ \Big(\frac{1}{|I|} \|f \cdot \chi_{I}^M\|_1\Big). \label{pr:size}
 \end{equation}
\end{itemize}

\begin{remark}
Si $a_I(f)=\vert \langle f, \phi_I \rangle\vert$, la condition \eqref{eq:energie} est une cons\'equence de la bornitude $L^2$ de la fonction carr\'ee 
\begin{equation}
\label{def:square-function}
S_{I_0}(f)(x):=\Big(  \sum_{I \subseteq I_0} \vert a_I(f) \vert^2 \frac{\one_I(x)}{\vert I \vert} \Big)^{1/2}.
\end{equation}
Aussi, la condition sur le ``size'' est une cons\'equence du th\'eor\`eme de John-Nirenberg, car le ``size'' se comporte comme la norme $BMO$. Cela 
devient plus \'evident dans le cas $a_I(f)=\langle f, h_I \rangle$ (o\`u $(h_I)_I$ est la base de Haar): on a alors
\[
\sum_{\substack{I \subseteq I_0 \\ I \in \Dy^0}} \vert a_I(f) \vert^2= \sum_{\substack{I \subseteq I_0 \\ I \in \Dy^0}} \vert \langle f, h_I \rangle \vert^2 =\big\| \big( f - \aver{I_0} f \big) \cdot \one_{I_0} \big\|_2^2
\]
et par le th\'eor\`eme de John-Nirenberg,
\[
\sup_{I \subseteq I_0} \frac{1}{\vert I \vert^{1/2}} \big\| \big( f - \aver{I} f \big) \cdot \one_{I} \big\|_2 \sim \sup_{I \subseteq I_0} \frac{1}{\vert I \vert} \big\| \big( f - \aver{I} f \big) \cdot \one_{I} \big\|_1.
\]
\end{remark}

Nous allons tout d'abord obtenir un contr\^ole \'epars classique de l'op\'erateur $T$, par la forme bilin\'eaire, ceci permettant d'en d\'eduire des estimations \`a poids $L^p(\omega)$ et le type faible $L^1$ de l'op\'erateur $T$ et de son dual $T^*$. Puis dans un second temps, nous allons montrer comment on peut conserver la structure 'fr\'equentielle' (d\'ecomposition en terme de coefficients localis\'es en espace et en fr\'equence) \`a travers un contr\^ole \'epars. Ceci nous permettra alors d'en d\'eduire et de quantifier de mani\`ere pr\'ecise des estimations dans des espaces de Hardy \`a poids.

\subsection{Domination \'eparse et estimation dans des espaces $L^p$ \`a poids pour $1<p<\infty$} \label{subsec:p>1}

\begin{theorem} \label{thm:p>1} Sous les hypoth\`eses \eqref{eq:energie} et \eqref{pr:size}, pour toutes fonctions $f,g\in L^2$ il existe une collection \'eparse $\mcS$ telle que
$$ \left| \langle Tf, g \rangle \right| \lesssim \sum_{Q\in \mcS} \left(\aver{Q} |f| \, dx \right) \left(\aver{Q} |g| \, dx \right) |Q|.$$ 
Les op\'erateurs $T$ et $T^*$ sont de type faible $L^1$ et v\'erifient pour $p\in(1,\infty)$
$$ \| T\|_{L^p(\omega) \to L^p(\omega)} + \| T^*\|_{L^p(\omega) \to L^p(\omega)} \lesssim [\omega]_{A_p}^{\max\{1,(p-1)^{-1}\}}.$$
\end{theorem}

Ce th\'eor\`eme est la cons\'equence de la Proposition \ref{prop:sparse-Hilbert-trans} ci-dessous, ainsi que du Corollaire \ref{coro:typefaible} avec la Proposition \ref{prop:enlarge_sparse}. 

On consid\`ere un mod\`ele de l'op\'erateur $T$, pour lequel la forme bilin\'eaire est donn\'ee par
\[
\Lambda_\ii I (f, g):=\sum_{I \in \ii I } \epsilon_I a_I(f) a_I(g)
\]
o\`u $\ii I$ est une collection finie des intervalles dyadiques, et $(\epsilon_I)_I$ est une suite born\'ee uniform\'ement (pour simplicit\'e, on suppose que $\vert \epsilon_I \vert \leq 1$ pour tout $I \in \ii I$). Les estimations obtenues ne d\'ependeront pas des coefficients ni de la collection $\ii I$. On rappelle que le prototype est celui de la forme bilin\'eaire $\Lambda_\ii I (f, g):=\sum_{I \in \ii I } \epsilon_I\langle f, \phi_I \rangle \langle g, \phi_I \rangle$, o\`u les $\phi_I$ sont des paquets d'onde $L^2$-normalis\'ees.

\begin{remark} Etant donn\'e que les fonctions $f,g$ sont fix\'ees dans $L^2$ et gr\^ace \`a \eqref{eq:energie}, on a la situation suivante: ou bien $\langle Tf, g \rangle=0$ et il n'y a rien \`a montrer, ou alors il existe une sous collection finie $\ii I \subset \Dy^0$ telle que
$$ \sum_{I\in \Dy^0 \setminus \ii I} |a_I(f)| \cdot |a_I(g)| \leq \frac{1}{2} |\langle Tf, g \rangle|.$$
Il suffit donc d'estimer seulement $\Lambda_{\ii I}(f,g)$ ce qui montre que l'on peut bien se restreindre \`a une collection finie d'intervalles.
\end{remark}

La propri\'et\'e la plus importante est le caract\`ere {\bf local} de la forme bilin\'eaire. On commence avec $I_0$ un intervalle dyadique fix\'e,  qui n'est pas n\'ecessairement contenu dans la collection $\ii I$. On d\'enote
\[
\ii I (I_0):=\left \lbrace I \in \ii I : I \subseteq I_0  \right\rbrace \quad \text{et  } \quad \ii I^+ (I_0):=\ii I (I_0) \cup \lbrace I_0 \rbrace.
\]

On a alors le r\'esultat suivant:

\begin{proposition}[Localisation de la forme bilin\'eaire]
\label{prop:local-Hilbert}
\[
\big|\Lambda_{\ii I \left(I_0 \right)}(f, g) \big| \lesssim \sssize_{I_0} f \cdot \sssize_{I_0} g \cdot |I_0|,
\]
o\`u le nouveau ``$\sssize$" est d\'efini par
\[
\sssize_{I_0}f := \sup_{ J \in \ii I^+ \left( I_0 \right)} \frac{1}{|J|} \int_{\rr R} |f| \cdot \ci_{J}^M dx.
\]
\end{proposition}

\begin{proof}
On peut estimer la forme bilin\'eaire, en appliquant Cauchy-Schwartz, par
\begin{align*}
\vert  \Lambda_{\ii I \left(I_0 \right)}(f, g)  \vert &\lesssim \big\|  \big( \sum_{I \in \ii I(I_0)}  \vert a_I(f) \vert^2 \frac{\one_I}{\vert I \vert} \big)^{1/2}\big\|_2 \cdot \big\|  \big( \sum_{I \in \ii I(I_0)}  \vert a_I(g) \vert^2 \frac{\one_I}{\vert I \vert} \big)^{1/2}\big\|_2 \\
& \lesssim \frac{1}{\vert I_0 \vert^{1/2}} \big\|  \big( \sum_{I \in \ii I(I_0)}  \vert a_I(f) \vert^2 \frac{\one_I}{\vert I \vert} \big)^{1/2}\big\|_2 \cdot \frac{1}{\vert I_0 \vert^{1/2}} \big\|  \big( \sum_{I \in \ii I(I_0)}  \vert a_I(g) \vert^2 \frac{\one_I}{\vert I \vert} \big)^{1/2}\big\|_2 \cdot \vert I_0 \vert \\
&\lesssim \ssize_{\ii I^+(I_0)} f \cdot \ssize_{\ii I^+ (I_0)} g \cdot |I_0|.
\end{align*}
Bien que $\ssize_{\ii I^+(I_0)}$ est une quantit\'e dans $L^2$, par John-Nirenberg et la bornitude $L^1 \mapsto L^{1, \infty}$ de la fonction carr\'ee, on peut remplacer $\ssize_{\ii I^+(I_0)}$ par $\sssize_{\ii I^+(I_0)}$, qui est une quantit\'e dans $L^1$. L'estimation $\ssize_{\ii I^+(I_0)}f \lesssim \sssize_{\ii I^+(I_0)} f$ est pr\'esent\'ee dans \cite[Lemma 2.13]{multilinear_harmonic}, \cite[Lemma 4.2]{MTTBiest0}.
\end{proof}

\begin{remark}
D'une mani\`ere g\'en\'erale, les intervalles $I_0$ qui seront consid\'er\'es ult\'erieurement, seront choisis par un temps d'arr\^et et v\'erifierons la propri\'et\'e que
$$ \sssize_{\ii I_{I_0}} f \lesssim \frac{1}{|I_0|} \int_{\rr R} |f| \cdot \ci_{I_0}^M dx $$
et l'analogue pour $g$. Autrement dit, le $\sssize$ qui est un supremum sur la collection $\ii I_{I_0}$ est control\'e par l'information \`a l'\'echelle de $I_0$.
Dans un tel cas, le $\sssize$ peut-\^etre vu comme une fonction maximale $\ic M_{\ii I \left( I_0 \right)}$. Ainsi, le temps d'arr\^et choisi les bons intervalles $I_0$ et  collections $\ii I_{I_0} \subseteq \ii I(I_0)$ tels qu'on a, en utilisant la Proposition \ref{prop:local-Hilbert}
$$|\Lambda_{\ii I _{ I_0} }(f, g)| \lesssim \inf_{y \in I_0}\ic M_{\ii I _{ I_0}}(f)(y) \cdot \inf_{y \in I_0}\ic M_{\ii I _{ I_0}}(g)(y) \cdot |I_0|.$$
\end{remark}

De fa\c{c}on similaire au Th\'eor\`eme \ref{thm:p>1}, on obtiendra:
\begin{proposition}
\label{prop:sparse-Hilbert-trans}
Il existe une collection \'eparse $\ic S$ d'intervalles  telle que
\begin{equation}
\label{eq:sparseHilbert}
\big|\Lambda_{\ii I}(f, g)\big| \lesssim  \sum_{Q \in \ic S}  \left(\frac{1}{|Q|} \int_{\rr R} |f|\cdot \ci_{Q}^M dx \right) \cdot \left(\frac{1}{|Q|} \int_{\rr R} |g|\cdot \ci_{Q}^M dx \right) \cdot |Q|.
\end{equation}
\end{proposition}

\begin{proof}
 Dans le cas pr\'esent, nous allons consid\'erer la formulation - d\'efinition \ref{def:sparsebis} pour le caract\`ere \'epars de la collection, avec un param\`etre $\eta=\frac{1}{2}$.

La collection $\ic S$ sera construite de mani\`ere r\'ecursive, de sorte qu'on peut l'\'ecrire comme $\ic S =\bigcup_k \ic S_k$. En effet, chaque intervalle $Q$ 
repr\'esente un bon support pour la forme bilin\'eaire localis\'ee $\Lambda_{\ii I}$, et il ne doit pas appartenir \`a la collection $\ii I$. Pour chaque $Q \in \ic S$ on aura une collection d'intervalles $\ii I_Q \subseteq \ii I (Q)$ associ\'ee telle que 
\[
\big| \Lambda_{\ii I_Q}(f, g)  \big| \lesssim \left(\frac{1}{|Q|} \int_{\rr R} |f|\cdot \ci_{Q}^M dx \right) \cdot \left(\frac{1}{|Q|} \int_{\rr R} |g|\cdot \ci_{Q}^M dx \right) \cdot |Q|.
\]

On commence par poser $\ii I_{Stock}:=\ii I$. On d\'efinit la collection $\ic S_0$ (le niveau z\'ero de la collection \'eparse) comme l'ensemble des intervalles maximaux $Q_0 \in \ii{I}$, et pour chacun d'entre eux, on d\'efinit $\ii I_{Q_0}$ comme la collection d'intervalles  $I \in \ii I_{Stock}(Q_0)$ telle que
\begin{equation}
\label{eq:condition}
\frac{1}{|I|} \int_{\rr R} |f| \cdot \ci_I^M dx \leq C \frac{1}{|Q_0|} \int_{\rr R} |f| \cdot \ci_{Q_0}^M dx \quad   \text{    et     } \quad \frac{1}{|I|} \int_{\rr R} |g| \cdot \ci_I^M dx \leq C \frac{1}{|Q_0|} \int_{\rr R} |g| \cdot \ci_{Q_0}^M dx.
\end{equation}

Apr\`es, on r\'einitialise $\ii I_{Stock}:= \ii I_{Stock} \setminus \bigcup_{Q_0 \in \ic S_0} \ii I_{Q_0}$, et il nous reste \`a d\'efinir les descendants de chaque $Q_0$. Les intervalles $I \in \ii I_{Stock}(Q_0)$, s'il y en a encore, ils satisfont une des deux conditions suivantes: soit
\[
 \frac{1}{|I|} \int_{\rr R} |f| \cdot \ci_I^M dx > C \frac{1}{|Q_0|} \int_{\rr R} |f| \cdot \ci_{Q_0}^M dx \quad \text{ ou  } \quad \frac{1}{|I|} \int_{\rr R} |g| \cdot \ci_I^M dx > C \frac{1}{|Q_0|} \int_{\rr R} |g| \cdot \ci_{Q_0}^M dx.
 \] 
L'ensemble des descendants de $Q_0$, $ch_{\ic S}(Q_0)$, se compose des intervalles maximaux $Q \subset Q_0$ qui contiennent au moins un intervalle $I \in \ii I_{Stock}(Q_0)$ (donc, la moyenne de $|f| \cdot \ci_{Q_0}$ sur $I$ est trop grande, ou celle de $|g| \cdot \ci_{Q_0}$), et pour lesquels on a 
\[
 \frac{1}{|Q|} \int_{\rr R} |f| \cdot \ci_Q^M dx > C \frac{1}{|Q_0|} \int_{\rr R} |f| \cdot \ci_{Q_0}^M dx \quad \text{ o\`u  } \quad \frac{1}{|Q|} \int_{\rr R} |g| \cdot \ci_Q^M dx > C \frac{1}{|Q_0|} \int_{\rr R} |g| \cdot \ci_{Q_0}^M dx.
 \] 

De cette fa\c{c}on, $ch_{\ic S}(Q_0)$ est une collection d'intervalles disjoints, dont l'union est contenue dans
\[
\big\lbrace x: \ic M \big( |f| \cdot \ci_{Q_0}^M \big) > C \frac{1}{|Q_0|}\int_{\rr R} |f| \cdot \ci_{Q_0}^M dx   \big\rbrace \cup \big\lbrace x: \ic M \big( |g| \cdot \ci_{Q_0}^M \big)> C \frac{1}{|Q_0|}\int_{\rr R} |g| \cdot \ci_{Q_0}^M dx    \big\rbrace.
\]
La bornitude $L^1 \to L^{1, \infty}$ de la fonction maximale entraine alors que $\ds \sum_{Q \in ch_{\ic S}(Q_0)}|Q| \leq \frac{1}{2}|Q_0|$, si la constante $C$ est choisie assez grande,  ce qui correspond \`a la condition ``\'eparse'' d\'esir\'ee. 

Avant de d\'ecrire la construction it\'erative  $(k \Rightarrow k+1)$, les observations suivantes sont n\'ecessaires:
\begin{itemize}
\item[(a)] Dans la construction de $\ic S_0$, chaque $Q_0$ est contenu dans $\ii I_{Q_0}$, donc celle-ci n'est pas vide. \c{C}a ne sera pas toujours le cas pour $k \geq 1$ (il est possible que $\ii I_{Q_0} =\emptyset$), mais \c{c}a n'aura aucune cons\'equence pour le temps d'arr\^et.
\item[(b)] On a $\sssize_{\ii I_{Q_0}} f \lesssim \frac{1}{|Q_0|} \int_{\rr R} |f| \cdot \ci_{Q_0}^M dx$ et de m\^eme pour $g$.
\end{itemize}

Supposons maintenant que $\ic S_k$ a \'et\'e d\'ejà d\'efinie, et que pour chaque $Q_0 \in \ic S_k$, on a une collection \'eparse des descendants directs $ch_{\ic S}(Q_0)$, ainsi qu'une collection $\ii I_{Q_0} \subset \ii I$. On rappelle que $\ds \ii I_{Stock}:=\ii I \setminus \bigcup_{j \leq k} \bigcup_{K \in \ic S_j} \ii I_K$ est constitue par les intervalles $I \in \ii I$ qui n'ont pas \'et\'e choisis par le temps d'arr\^et. 

On d\'efini alors $\ic S_{k+1}:= \cup_{Q_0 \in \ic S_k} \lbrace Q: Q \in ch_{\ic S}(Q_0) \rbrace$, et il reste \`a d\'efinir $\ii I_Q$ et $ch_{\ic S}(Q)$ pour chaque $Q \in \ic S_{k+1}$.
On cherche des intervalles $I \in \ii I_{Stock} \cap \ii I(Q)$ telle que 
\begin{equation}
\label{eq:condition-k}
\frac{1}{|I|} \int_{\rr R} |f| \cdot \ci_I^M dx \leq C \frac{1}{|Q|} \int_{\rr R} |f| \cdot \ci_{Q}^M dx \quad   \text{    et     } \quad \frac{1}{|I|} \int_{\rr R} |g| \cdot \ci_I^M dx \leq C \frac{1}{|Q|} \int_{\rr R} |g| \cdot \ci_{Q}^M dx.
\end{equation}

Il est possible qu'il n'y a pas d'intervalles satisfaisant cette propri\'et\'e, donc pour $k \geq 1$ il peut arriver que $\ii I_Q =\emptyset$. Apr\`es avoir construit la collection $\ii I_Q$, on r\'einitialise $\ds\ii I_{Stock}:=\ii I_{Stock} \setminus \bigcup_{Q \in \ic S_{k+1}} \ii I_Q$.

Les descendants directs $ch_{\ic S}(Q)$ de $Q$ sont des intervalles maximaux $P$ qui contiennent au moins un  $I \in \ii I_{Stock} \cap \ii I(Q)$ (pour un tel $I$, la condition \eqref{eq:condition-k} est fausse), et de plus, on veut aussi que 
\[
 \frac{1}{|P|} \int_{\rr R} |f| \cdot \ci_P^M dx > C \frac{1}{|Q|} \int_{\rr R} |f| \cdot \ci_{Q}^M dx \quad \text{ o\`u  } \quad \frac{1}{|P|} \int_{\rr R} |g| \cdot \ci_P^M dx > C \frac{1}{|Q|} \int_{\rr R} |g| \cdot \ci_{Q}^M dx.
 \]

Comme on l'a d\'ejà vu, la bornitude de la fonction maximale implique que $\sum_{P \in ch_{S}\left(Q\right)}|P| \leq \frac{1}{2}|Q|$. Si on arrive \`a avoir 
$\ii I_{Stock} \cap \ii I(Q)=\emptyset$, alors on choisi $ch_{\ic S}(Q)=\emptyset$.

La proc\'edure s'arr\^etera apr\`es un nombre fini d'\'etapes car la collection initiale $\ii I$ \'etait finie. Pour chaque $I \in \ii I$ il existe un unique 
$Q \in \cup_k \ic S_k$ telle que $I$ est contenu dans $\ii I_Q$: on peut voir qu'il existe un certain $Q$ (qui doit \^etre choisi par le temps d'arr\^et) pour lequel
\[
\frac{1}{|I|} \int_{\rr R} |f| \cdot \ci_I^M dx \leq C \frac{1}{|Q|} \int_{\rr R} |f| \cdot \ci_{Q}^M dx \quad   \text{    et     } \quad \frac{1}{|I|} \int_{\rr R} |g| \cdot \ci_I^M dx \leq C \frac{1}{|Q|} \int_{\rr R} |g| \cdot \ci_{Q}^M dx.
\]
La propri\'et\'e ci-dessus est vraie pour $Q=I$, mais \c{c}a peut \^etre vrai avec un intervalle $Q \in \ic S$ dont $I$ est strictement contenu.

De plus, le processus de s\'election implique que pour chaque $Q_0 \in \ic S$, on a
\[
\sssize_{\ii I_{Q_0}} f \lesssim \frac{1}{|Q_0|} \int_{\rr R} |f| \cdot \ci_{Q_0}^M dx \quad \text{et} \quad \sssize_{\ii I_{Q_0}} g \lesssim \frac{1}{|Q_0|} \int_{\rr R} |g| \cdot \ci_{Q_0}^M dx.
\]
Cela permet d'estimer la forme bilin\'eaire par:
\begin{align*}
\big| \Lambda_{\ii I }(f, g)\big| &\lesssim \sum_k \sum_{Q \in S_k}  \big| \Lambda_{\ii I_Q}(f,g) \big| \\
&\lesssim \sum_{Q \in \ic S} \sssize_{\ii I_Q} f \cdot \sssize_{\ii I_Q} g \cdot |Q| \\
&\lesssim \sum_{Q \in \ic S} \left(\frac{1}{|Q|} \int_{\rr R} |f| \cdot \ci_{Q}^M dx\right) \cdot \left(  \frac{1}{|Q|} \int_{\rr R} |g| \cdot \ci_{Q}^M dx\right) \cdot |Q|,
\end{align*}
qui correspond \`a l'estimation \'eparse de \eqref{eq:sparseHilbert}.
\end{proof}

\subsection{Domination \'eparse via des fonctions carr\'ees}
\label{subsec:carre}

Pour les multiplicateurs de Fourier, qui sont plus r\'eguliers que les op\'erateurs de Calder\'on-Zygmund g\'en\'eraux, on peut montrer une domination \'eparse avec des fonctions carr\'ees. C'est le but de cette sous-section, d'obtenir cette estimation plus pr\'ecises, qui conserve la structure fr\'equentielle de l'op\'erateur initial.

Comme pr\'ec\'edemment, on consid\`ere la forme bilin\'eaire mod\`ele:
\begin{equation}
\label{def:operator}
\Lambda_\ii I (f, g):=\sum_{I \in \ii I } c_I a_I(f) a_I(g)
\end{equation}
o\`u $\ii I \subset \Dy^0$ est une collection finie d'intervalles dyadiques, $\lbrace  c_I \rbrace_I$ est une suite de nombres complexes uniform\'ement born\'es.

On va montrer alors le r\'esultat suivant:

\begin{theorem} \label{thm:sparsecarre} Fixons des exposants $p,q\in(0,\infty)$. Pour toutes fonctions $f,g \in L^2$ il existe une collection \'eparse $\mcS$ telle que 
$$ \left|\Lambda_\ii I (f, g)\right| \lesssim \sum_{Q\in \mcS} \left(\aver{Q} |S_Qf|^p  dx \right)^{1/p} \cdot \left(\aver{Q} |S_Qg|^q  dx \right)^{1/q} \cdot |Q|.$$
\end{theorem}


\begin{proposition}
\label{prop:localization}
Si $I_0$ est un intervalle dyadique fix\'e et si  $\ii I(I_0):=\lbrace I \in \ii I : I \subseteq I_0  \rbrace$, alors
{\fontsize{10.5}{10}
\[
\big| \Lambda_{\ii I\left( I_0 \right)}(f,g) \big| \lesssim \sup_{I' \in \ii I \left( I_0 \right)} \frac{1}{\vert I' \vert^{\frac{1}{p}}} \big\|  \Big( \sum_{\substack{I \subseteq I' \\ I \in \ii I(I_0)}}  \frac{\vert a_I(f) \vert^2}{\vert I \vert}  \one_I \Big)^{1/2}   \big\|_{p, \infty} \cdot  \sup_{I' \in \ii I \left( I_0 \right)} \frac{1}{\vert I' \vert^{\frac{1}{q}}} \big\|  \Big( \sum_{\substack{I \subseteq I' \\ I \in \ii I(I_0)}}  \frac{\vert a_I(g) \vert^2}{\vert I \vert}  \one_I \Big)^{1/2} \big\|_{q, \infty} \cdot \vert I_0  \vert
\]}
pour tous $0<p, q <\infty$. 
\begin{proof}
L'in\'egalit\'e de Cauchy-Schwartz implique imm\'ediatement l'estimation 
\[
\big| \Lambda_{\ii I\left( I_0 \right)}(f,g) \big| \lesssim \frac{1}{\vert I_0 \vert^{\frac{1}{p}}} \big\|  \Big( \sum_{ I \in \ii I(I_0)}  \frac{\vert a_I(f) \vert^2}{\vert I \vert}  \one_I \Big)^{1/2}   \big\|_{2} \cdot \frac{1}{\vert I_0\vert^{\frac{1}{p}}} \big\|  \Big( \sum_{ I \in \ii I(I_0)}  \frac{\vert a_I(g) \vert^2}{\vert I \vert}  \one_I \Big)^{1/2}   \big\|_{2} \cdot \vert Q_0 \vert,
\]
qui entra\^ine la conclusion pour $p=q=2$, avec des norme $L^2$. Pour obtenir le cas g\'en\'eral pour tous $0<p, q <\infty$, il suffi d'invoquer le lemme de John-Nirenberg, comme pr\'esent\'e dans \cite[Section 2.6]{multilinear_harmonic}.
\end{proof}
\end{proposition}

\begin{proof}[Preuve du Th\'eor\`eme \ref{thm:sparsecarre}]
L'objectif est d'obtenir une collection \'eparse $\ic S$ d'intervalles dyadiques qui satisfont les conditions suivantes:
\begin{itemize}
\item pour chaque $Q \in \ic S$, il existe une collection $ch_{\ic S}(Q)$ des descendants directs telle que la condition \'eparse est v\'erifi\'ee: 
\[
\sum_{P \in h_{\ic S}(Q)} \vert P \vert \leq \frac{1}{2} \vert Q \vert.
\]
\item \'egalement, pour chaque $Q \in \ic S$ il existe une sous-collection $\ii I_Q \subseteq \ii I(Q)$ compos\'ee par de ``bons intervalles'' telle que 
\[
\vert\Lambda_{\ii I_Q}(f, g)\vert \lesssim \frac{1}{\vert Q \vert^{\frac{1}{p}}} \big\|  \Big( \sum_{ I \in \ii I(Q)}  \frac{\vert a_I(f) \vert^2}{\vert I \vert}  \one_I \Big)^{1/2}   \big\|_{p} \cdot \frac{1}{\vert Q \vert^{\frac{1}{q}}} \big\|  \Big( \sum_{ I \in \ii I(Q)} \frac{\vert a_I(g) \vert^2}{\vert I \vert}  \one_I \Big)^{1/2}   \big\|_{q} \cdot \vert Q  \vert.
\]
\end{itemize}

On commence en initialisant $\ii I_{Stock}:=\ii I$ et en choisissant $Q_0 \in \ii I_{Stock}$ un intervalle maximal (ceux-ci forment $\ic S_0$). La collection $\ii I_{Q_0}$ se compose par les intervalles $I' \in \ii I_{Stock}(Q_0)$ telle que
\begin{equation}
\label{eq:constraint-f}
\frac{1}{\vert I' \vert^{\frac{1}{p}}} \big\|  \Big( \sum_{ I \in \ii I_{Stock}(I')}  \frac{\vert a_I(f) \vert^2}{\vert I \vert}  \one_I \Big)^{1/2}   \big\|_{p} \leq C \cdot \frac{1}{\vert Q_0 \vert^{\frac{1}{p}}} \big\|  \Big( \sum_{ I \in \ii I_{Stock}(Q_0)}  \frac{\vert a_I(f) \vert^2}{\vert I \vert}  \one_I \Big)^{1/2}   \big\|_{p}
\end{equation}
et simultan\'ement pour la fonction $g$
\begin{equation}
\label{eq:constraint-g}
\frac{1}{\vert I' \vert^{\frac{1}{q}}} \big\|  \Big( \sum_{ I \in \ii I_{Stock}(I')}  \frac{\vert a_I(g) \vert^2}{\vert I \vert}  \one_I \Big)^{1/2}   \big\|_{q} \leq C \cdot \frac{1}{\vert Q_0 \vert^{\frac{1}{q}}} \big\|  \Big( \sum_{ I \in \ii I_{Stock}(Q_0)}  \frac{\vert a_I(g) \vert^2}{\vert I \vert}  \one_I \Big)^{1/2}   \big\|_{q}.
\end{equation}

Pour chaque $Q_0 \in \ic S$, on choisit ses descendants parmi les intervalles  $Q \in \ii I_{Stock}(Q_0)$ pour lesquels une des conditions \eqref{eq:constraint-f} ou \eqref{eq:constraint-g} est fausse, et qui, de plus, sont maximaux. Donc pour un tel intervalle $Q$ on a 
\[
\frac{1}{\vert Q \vert^{\frac{1}{p}}} \big\|  \Big( \sum_{ I \in \ii I_{Stock}(Q)}  \frac{\vert a_I(f) \vert^2}{\vert I \vert}  \one_I \Big)^{1/2}   \big\|_{p} > C \cdot \frac{1}{\vert Q_0 \vert^{\frac{1}{p}}} \big\|  \Big( \sum_{ I \in \ii I_{Stock}(Q_0)}  \frac{\vert a_I(f) \vert^2}{\vert I \vert}  \one_I \Big)^{1/2}   \big\|_{p},
\]
ou la condition similaire pour $g$. Dans les deux cas, l'in\'egalit\'e au-dessus implique que 
\[
\inf_{y \in Q} \ic M_p \big( \big( \sum_{ I \in \ii I_{Stock}(Q_0)}  \frac{\vert a_I(f) \vert^2}{\vert I \vert}  \one_I \big)^{1/2}  \big)(y ) >C \cdot \frac{1}{\vert Q_0 \vert^{\frac{1}{p}}} \big\|  \big( \sum_{ I \in \ii I_{Stock}(Q_0)}  \frac{\vert a_I(f) \vert^2}{\vert I \vert}  \one_I \big)^{1/2}   \big\|_{p}.
\]
Maintenant on peut  v\'erifier la condition \'eparse: $\sum_{Q \in h_{\ic S}(Q_0)} \vert Q \vert \leq \frac{1}{2} \vert Q_0 \vert$. On note que chaque $Q$ est contenu dans l'ensemble
\[
 \Big\lbrace  x: \ic M_p \big( \sum_{ I \in \ii I_{Stock}(Q_0)}  \frac{\vert a_I(f)\vert^2}{\vert I \vert}  \one_I(x) \big)^{1/2} > C  \frac{1}{\vert Q_0 \vert^{\frac{1}{p}}} \big\|  \big( \sum_{ I \in \ii I_{Stock}(Q_0)}  \frac{\vert a_I(f) \vert^2}{\vert I \vert}  \one_I \big)^{1/2}   \big\|_{p} \Big\rbrace, 
\]
o\`u dans l'ensemble correspondant \`a la fonction $g$.

On utilise la bornitude $L^p \to L^{p, \infty}$ de l'op\'erateur $\ic M_p$ et la maximalit\'e des intervalles $Q \in ch_{\ic S}(Q_0)$ pour calculer
{\fontsize{10.5}{10}
\begin{align*}
\sum_{Q \in ch_{\ic S}\left( Q_0 \right) } \vert Q \vert &\lesssim \Big(  C  \frac{1}{\vert Q_0 \vert^{\frac{1}{p}}} \big\|  \big( \sum_{ I \in \ii I_{Stock}(Q_0)}  \frac{\vert a_I(f) \vert^2}{\vert I \vert}  \one_I \big)^{1/2}   \big\|_p \Big)^{-p} \big\| \ic M_p \big( \sum_{ I \in \ii I_{Stock}(Q_0)}  \frac{\vert a_I(f) \vert^2}{\vert I \vert}  \one_I(x) \big)^{1/2} \big\|_{p, \infty}^p \\
& \lesssim C^{-1} \vert Q_0\vert.
\end{align*}}
En effet, on doit prendre en consid\'eration les intervalles $Q \in ch_{\ic S}(Q_0)$ pour lesquels la condition \eqref{eq:constraint-g} est fausse pour la fonction $g$. En choisissant une constante $C$ assez grande, on obtient la condition \'eparse.

Apr\`es avoir construit les collections $\ii I_{Q_0}$ et $ch_{\ic S}(Q_0)$, on r\'einitialise $\ii I_{Stock}:=\ii I_{Stock} \setminus \cup_{Q_0 \in \ic S_0} \ii I_{Q_0}$. D'ici, on r\'eit\`ere la proc\'edure d'une mani\`ere claire pour obtenir une collection \'eparse $\ic S$.

Cela permet d'estimer la forme bilin\'eaire par 
\begin{equation}
\label{eq:sparce-sf}
\vert \Lambda_{\ii I} (f, g) \vert \lesssim \sum_{Q \in \ic S} \frac{1}{\vert Q \vert^{\frac{1}{p}}} \big\|  \Big( \sum_{ I \in \ii I(Q)}  \frac{\vert a_I(f) \vert^2}{\vert I \vert}  \one_I \Big)^{1/2}   \big\|_p \cdot \frac{1}{\vert Q \vert^{\frac{1}{p}}} \big\|  \Big( \sum_{ I \in \ii I(Q)} \frac{\vert a_I(g) \vert^2}{\vert I \vert}  \one_I \Big)^{1/2}   \big\|_q \cdot \vert Q  \vert,
\end{equation}
pour tous $0<p, q<\infty$.
\end{proof}

\subsection{Application \`a la bornitude dans des espaces de Hardy $H^p$ \`a poids pour $0<p\leq 1$}
\label{subsec:p<1}

Pour $I_0$ un intervalle, on consid\`ere la fonction carr\'ee localis\'ee
$$ S_{I_0}(f) := x \mapsto \big( \sum_{I \subset I_0} |a_I(f)|^2 \frac{\one_I(x)}{|I|}\big)^{1/2}.$$

Pour un poids $\omega\in L^1_{loc}$, on d\'efini l'espace de Hardy \`a poids $H^p_\omega$ pour $p\in(0,1]$ par la norme $\|Sf \|_{L^p_\omega}$ et l'espace $CMO_\omega^p$ par la norme
$$ \sup_{I_0} \frac{1}{\omega(I_0)^\frac{1}{p}} \Big(\omega(I_0) \sum_{I \subseteq I_0} \vert a_I(f)\vert^2 \frac{\vert I \vert}{\omega(I)} \Big)^{1/2}.$$
Ici, on identifie le poids $\omega$ avec la mesure $\omega dx$, pour noter $\omega(I)=\int_I \omega(x) dx$. Nous allons montrer l'estimation suivante:

\begin{proposition} \label{prop:Hardy} Sous les hypoth\`eses \eqref{eq:energie} et \eqref{pr:size}, fixons un exposant $p\in(0,1]$, un poids $\omega$ et deux fonctions $f\in H^p_\omega$ et $g\in CMO_\omega^p$. Alors pour toute collection finie $\ii I$ de cubes
$$ \vert \Lambda_{\ii I}(f, g) \vert \lesssim \| f \|_{H^p_\omega} \cdot \| g \|_{CMO_\omega^p},$$
avec une constante implicite ind\'ependante des fonctions $f,g$ et du poids $\omega$.
\end{proposition}

\begin{proof} On suit la m\^eme approche que celle de la sous-section pr\'ec\'edente. On a d'abord une estimation \`a poids localis\'ee. Si $I_0$ est un intervalle fix\'e alors
\begin{align*}
 \vert \Lambda_{I_0} (f,g) \vert& \lesssim  \sum_{I \subset I} \vert a_I(f)\vert \cdot \vert a_I(g)\vert=\sum_{I \subseteq I_0} \vert a_I(f)\vert \frac{\omega(I)^{\frac{1}{2}}}{\vert I \vert} \cdot \vert a_I(g) \vert \frac{\vert I \vert^{\frac{1}{2}}}{\omega(I)^{\frac{1}{2}}} \\
              &\lesssim \frac{1}{\omega(I_0)^{\frac{1}{2}}} \big\|  \big( \sum_{I \subseteq I_0} \vert a_I(f) \vert^2 \cdot \frac{\one_I}{\vert I \vert}  \big)^{\frac{1}{2}} \big\|_{L^2\left(\omega \right)} \cdot \frac{1}{\omega(I_0)^\frac{1}{p}} \Big(\omega(I_0) \sum_{I \subseteq I_0} \vert a_I(f)\vert^2 \frac{\vert I \vert}{\omega(I)} \Big)^{1/2} \cdot \omega(I_0)^\frac{1}{p}\\
              & \lesssim \frac{1}{\omega(I_0)^{\frac{1}{2}}} \big\|  \big( \sum_{I \subseteq I_0} \vert a_I(f) \vert^2 \cdot \frac{\one_I}{\vert I \vert}  \big)^{\frac{1}{2}} \big\|_{L^2\left(\omega \right)} \cdot \|g\|_{CMO_\omega^p} \cdot \omega(I_0)^{\frac{1}{p}} \\
              &\lesssim \sup_{\tilde I \subseteq I_0} \frac{1}{\omega(\tilde I)^{\frac{1}{2}}} \big\|  \big( \sum_{I \subseteq  \tilde I} \vert a_I(f) \vert^2 \cdot \frac{\one_I}{\vert I \vert}  \big)^{\frac{1}{2}} \big\|_{L^2\left(\omega \right)}  \cdot \|g\|_{CMO_\omega^p}\cdot \omega(I_0)^{\frac{1}{p}}. 
\end{align*}
Gr\^ace au th\'eor\`eme de John-Nirenberg, qui reste vrai m\^eme si on ajoute un poids (la preuve et identique \`a celui de Th\'eor\`eme 2.7 du \cite{multilinear_harmonic}), on obtient
\begin{align*}
 \vert \Lambda_{I_0} (f,g) \vert& \lesssim \sup_{ \tilde I} \frac{1}{\omega(\tilde I)^{\frac{1}{p}}} \big\|  \big( \sum_{I \subseteq \tilde I} \vert a_I(f) \vert^2 \cdot \frac{\one_I}{\vert I \vert}  \big)^{\frac{1}{2}} \big\|_{L^p\left(\omega \right)}  \cdot \|g\|_{CMO_\omega^p}\cdot \omega(I_0)^{\frac{1}{p}}. 
\end{align*}
D\`es qu'on a ces estimations localis\'es, on peut refaire le m\^eme temps d'arr\^et pour obtenir une domination $\omega$-\'eparse (\'eparse relative \`a la mesure $\omega dx$). En effet, la condition dans le temps d'arr\^et pour choisir $\ii I_{Q_0}$ est donn\'ee par
\begin{equation*}
\frac{1}{\omega\left( I' \right)^{\frac{1}{r}}} \big\|  \big( \sum_{\substack{ I \subseteq I' \\ I \in \ii I_{Stock}}} \vert a_I(f) \vert^2  \frac{1_I}{\vert I \vert}   \big)^{\frac{1}{2}} \big\|_{L^r\left(\omega \right)} \leq C \cdot \frac{1}{\omega\left( Q_0\right)^{\frac{1}{r}}} \big\|  \big( \sum_{\substack{ I \subseteq Q_0 \\ I \in \ii I_{Stock}}} \vert a_I(f) \vert^2  \frac{1_I}{\vert I \vert}   \big)^{\frac{1}{2}} \big\|_{L^r\left(\omega \right)},
\end{equation*}
avec un $r<p$ qui sera choisi ult\'erieurement.

L'estimation $\omega$-\'eparse qu'on obtient \`a la fin est 
\begin{align*}
\vert \Lambda_{\ii I}(f, g) \vert \lesssim \sum_{Q \in \ic S} \frac{1}{\omega\left( Q \right)^{\frac{1}{r}}} \big\|  \big( \sum_{I \in \ii I_Q} \vert a_I(f) \vert^2  \frac{1_I}{\vert I \vert}   \big)^{\frac{1}{2}} \big\|_{L^r\left(\omega \right)} \cdot \big\| g \big\|_{CMO_\omega^p} \cdot \omega(Q)^{\frac{1}{p}}.
\end{align*}
On fait alors appara\^itre l'op\'erateur maximal \`a poids $\ic M_{r, \omega}$ d\'efini par 
\[
\ic M_{r, \omega} f(x):= \left( \ic M_\omega \vert f \vert^r  \right)^{1/r},
\]
qui est born\'e sur $L^p(\omega)$ uniform\'ement en fonction du poids $\omega$, car $\ic M_\omega$ est born\'e sur $L^{p/r}(\omega)$ (voir \cite[Th\'eor\`eme 15.1]{LernerNazarov} ou \cite{BernicotFreyPetermichl}).

En utilisant que la collection $\ic S$ est $\omega$-\'eparse, pour chaque $Q \in \ic S$ on peut trouver $E(Q) \subseteq Q$ telle que $\omega(E(Q)) > \frac{1}{2} \omega(Q)$ et tels que les ensembles $\lbrace E(Q)  \rbrace_{Q \in \ic S}$ sont deux \`a deux disjoints. Avec ces observations, on peut estimer la forme bilin\'eaire par
\begin{align*}
\vert \Lambda_{\ii I}(f, g) \vert & \lesssim \inf_{y \in E(Q)} \ic M_{r, \omega} \big( \sum_{I \in \ii I_Q} \vert a_I(f) \vert^2  \frac{1_I}{\vert I \vert}   \big)^{\frac{1}{2}}(y) \cdot \| g \|_{CMO_\omega^p} \cdot \omega(E(Q))^{\frac{1}{p}} \\
& \lesssim  \big(\int_{\rr R} \vert \ic M_{r, \omega} \big( \big( \sum_{I \in \ii I_Q} \vert a_I(f) \vert^2  \frac{1_I}{\vert I \vert}   \big)^{\frac{1}{2}}  \big)  \vert^p \omega(x) dx\big)^{\frac{1}{p}} \cdot \| g \|_{CMO_\omega^p} \\
&\lesssim \| Sf \|_{L^p\left( \omega \right)} \cdot \| g \|_{CMO_\omega^p}.
\end{align*}
\end{proof}

\begin{remark}
On observe que le temps d'arr\^et qu'on a utilise est vraiment similaire \`a celui qui sera utilis\'e ult\'erieurement dans la sous-section \ref{sec:atomic-dec}. La condition pour le temps d'arr\^et et toujours donn\'ee par une inspection de la moyenne de la fonction carr\'ee dans $L^r$, avec $r<p$.   
\end{remark}

D'apr\`es \cite{Lee-Lin-Lin}, $CMO_\omega^p$ est l'\'espace dual de $H^p_\omega$, de mani\`ere ind\'ependante du poids. En fait, les \'enonc\'es dans \cite{Lee-Lin-Lin} sont d\'ecrits avec un poids $\omega\in A_\infty$ car l'espace de Hardy est d\'efini en termes d'une fonction maximale. Si on choisit (comme c'est le cas ici) la norme de l'espace de Hardy d\'efinie en terme de la fonction carr\'ee alors, les r\'esultats de \cite{Lee-Lin-Lin} d\'ecrivent la dualit\'e de mani\`ere ind\'ependante du poids.
On obtient alors le r\'esultat suivant:

\begin{theorem} \label{thm:hardy} Soit $T$ un op\'erateur v\'erifiant les hypoth\`eses \eqref{eq:energie} et \eqref{pr:size}. Pour tout exposant $p\in(0,1]$, il existe une constante $C=C(T,p)$ telle que pour tout poids $\omega$, $T$ est continu sur $H^p_\omega$ avec
$$ \|T\|_{H^p_\omega \to H^p_\omega} \lesssim C.$$
\end{theorem}



\section{Application \`a la transform\'ee de Riesz et au projecteur de Leray}
\label{sec:leray}

Sur l'espace Euclidien ${\rr R}^n$, consid\'erons l'op\'erateur du second ordre $L=-\dive A \nabla$, avec les hypoth\`eses suivantes:

\begin{assum-L}
\begin{itemize}
\item[$\bullet$] Soit $A = A(x)$ une application \`a valeurs matricielles complexes, d\'efinie sur ${\rr R}^n$ et v\'erifiant une condition d'accr\'etivit\'e (ellipticit\'e)
\begin{equation} \lambda_1 |\xi|^ 2 \leq  \Re \langle A(x) \xi, \xi \rangle \qquad \textrm{and} \qquad  |\langle A(x)\xi , \zeta\rangle | \leq \lambda_2 |\xi||\zeta|, \label{eq:ell} \end{equation}
pour des constantes num\'eriques $\lambda_1,\lambda_2>0$ et tout $x\in {\rr R}^n$, $\xi,\zeta\in {\rr R}^n$. On y associe alors l'op\'erateur du second ordre d\'efini par
$$ L=L_A f :=- \dive (A\nabla \cdot ).$$
On sait que $L$ est un op\'erateur injectif, sectoriel et maximal accretif sur $L^2$, et donc admet un calcul fonctionnel holomorphe born\'e $H^\infty$ sur $L^2$. De plus $-L$ est le g\'en\'erateur d'un semi-groupe $(e^{-tL})_{t>0}$ sur $L^2$, qui v\'erifie des estimations de Davies-Gaffney $L^2$-$L^2$ (ainsi que $ \sqrt{t} e^{-tL} \dive$).

\item[$\bullet$] On suppose que le semi-groupe $e^{-tL}$ a une repr\'esentation int\'egrale par un noyau $K_t$ tel que ce noyau et son gradient v\'erifient des estimations Gaussiennes ponctuelles: pour tout $t>0$, 
\begin{equation} \left|K_t(x,y)\right| +  \sqrt{t} \left| \nabla_{x,y} K_t(x,y)\right| \lesssim |B(x,\sqrt{t})|^{-1} e^{-c \frac{|x-y|^2}{t} }
\label{eq:gaussian}
\end{equation}
avec des constantes implicites uniformes en $t>0$ et $x,y\in {\rr R}^n$.
\end{itemize}
\end{assum-L}

\begin{remark} Tous les r\'esultats de cette section pourraient \^etre \'ecrits dans le contexte d'une vari\'et\'e Riemannienne doublante non-born\'ee $(M,g)$ munie de son semi-groupe de la chaleur $(e^{-t\Delta})_{t>0}$ (o\`u $\Delta$ est l'op\'erateur positif de Laplace), sous l'hypoth\`ese que le semi-groupe et son gradient admettent une repr\'esentation int\'egrale avec un noyau, v\'erifiant des estimations ponctuelles Gaussiennes.
\end{remark}

Dans un tel cadre (ou celui d'une vari\'et\'e Riemanienne avec son semi-groupe de la chaleur), on d\'efinit la transform\'ee de Riesz $R_L:=\nabla L^{-1/2}$ et le projecteur de Leray $ \pi_L(f) := A \nabla L^{-1} \dive$.
Nous allons nous int\'eresser \`a l'\'etude de ces op\'erateurs et utiliser les r\'esultats pr\'ec\'edents sur le contr\^ole \'epars pour obtenir de nouveaux r\'esultats sur les bornitudes de ces op\'erateurs.

\medskip

Pr\'ecisons ici que le projecteur de Leray $\pi_L$  ou le dual de la transform\'ee de Riesz $R_L^*$ agissent sur un espace de champs de vecteurs. Pour les espaces de Lebesgue et un poids $\omega$, on doit donc distinguer l'espace $L^p_\omega= L^p_\omega({\rr R}^n; {\rr R})$ de fonctions et $L^p_\omega=L^p_\omega({\rr R}^n;{\rr R}^n)$ l'espace de champs de vecteurs, o\`u les deux espaces $L^p$ sont d\'efinis par la norme
$$ \left(\int_{{\rr R}^n} |f(x)|^p \, \omega(x) dx \right)^{1/p}$$
avec $|\cdot|$ \'etant le module ou une norme de ${\rr R}^n$ (selon que $f$ est une fonction ou un champ de vecteur. \\
Par simplicit\'e, on gardera la m\^eme notation $L^p$ de ces espaces, ainsi par exemple
$$ \| \pi_L\|_{L^p \to L^p} := \| \pi_L\|_{L^p_\omega({\rr R}^n;{\rr R}^n) \to L^p_\omega({\rr R}^n; {\rr R})}. $$
Pour diff\'erentier, nous allons prendre dans cette section, la notation suivante: une lettre majuscule (resp. minuscule) pour un champ de vecteurs (resp. fonction) et alors implicitement sa norme $L^p_\omega$ correspond \`a l'espace $L^p_\omega=L^p_\omega({\rr R}^n;{\rr R}^n)$ (resp. $=L^p_\omega({\rr R}^n;{\rr R})$).

\subsection{Type faible $L^1$ pour le dual de la transform\'ee de Riesz et autres fonctionnelles carr\'ees verticales}

Sous les hypoth\`eses faites sur $L$, les transform\'ees de Riesz $R_L:=\nabla L^{-1/2}$ et $R_{L^*}:=\nabla (L^*)^{-1/2}$ sont continues sur tous les espaces $L^p$ pour $p\in(1,\infty)$ et sont de type faible $L^1$ (voir par exemple \cite{Auscher-memoire}). Par dualit\'e, les op\'erateurs adjoint $R_L^* = (L^*)^{-1/2} \dive$ et $(R_{L^*})^* = L^{-1/2} \dive$ (agissant sur les champs de vecteurs) sont aussi continus sur $L^p$ pour $p\in(1,\infty)$.
Sous les conditions ci-dessus, il est montr\'e dans \cite{BernicotFreyPetermichl} que les transform\'ees de Riesz admettent un contr\^ole \'epars sous la forme suivante:

\begin{proposition}[\cite{BernicotFreyPetermichl}] Pour toutes fonctions $f,g\in L^2$ il existe une collection \'eparse $\mcS$ telle que
$$ \left| \langle R_{L^*}(f), g\rangle \right| + \left| \langle R_L(f), g\rangle \right| \lesssim \sum_{Q\in \mcS} \left(\aver{3Q} |f|\, dx \right) \left(\aver{3Q} |g|\, dx \right) |Q|. $$
\end{proposition}

En appliquant le Corollaire \ref{coro:typefaible}, on obtient donc le r\'esultat suivant:

\begin{proposition} \label{prop:dualriesz} Sous les hypoth\`eses faites sur $L$, les op\'erateurs adjoints des transform\'ees de Riesz $R_L^*$ et $ (R_{L^*})^*$ sont de type faible $L^1$: pour tout champ de vecteur $F\in L^1$
$$ \|(L^*)^{-1/2} \dive F\|_{L^{1,\infty}} + \|(L)^{-1/2} \dive F\|_{L^{1,\infty}} \lesssim \|F\|_{L^1}.$$
\end{proposition}

On obtient aussi le m\^eme r\'esultat pour les fonctions carr\'e associ\'ees: 
\begin{proposition} \label{prop:dualcarre} Sous les hypoth\`eses faites sur $L$, on fixe $Q_t(L)=(tL)^Ne^{-tL}$ pour un certain $N>1/2$. Alors les fonctions carr\'e
$$ \left(\int_0^\infty \left| \sqrt{t} Q_t(L) \dive(\cdot) \right|^2 \, \frac{dt}{t} \right)^{1/2} \quad \textrm{et} \quad \left(\int_0^\infty \left| \sqrt{t} Q_t(L^*) \dive(\cdot) \right|^2 \, \frac{dt}{t} \right)^{1/2}$$
sont de type faible $L^1$ et born\'ees sur $L^p$ pour $p\in(1,\infty)$.
\end{proposition}

\begin{proof} La bornitude $L^p$ pour $p\in(1,\infty)$ est bien connue par dualit\'e. On va donc se concentrer sur le type faible $L^1$, qui est nouveau et n\'ecessite une preuve plus compliqu\'ee. On ne traite par sym\'etrie que le cas de l'op\'erateur $L$. On commence par lin\'eariser la fonctions carr\'e, en utilisant $\epsilon_t:[0,1] \to [0,1]$ un bon syst\`eme de fonctions de Rademacher, on sait que
$$ \left(\int_0^\infty \left| \sqrt{t} Q_t(L) \dive(F)  \right|^2 \, \frac{dt}{t} \right)^{1/2} 
= \E_\omega\left[ \left|\int_0^\infty \epsilon_t(\omega) \sqrt{t} Q_t(L) \dive(F) \, \frac{dt}{t} \right| \right]$$
Fixons $F\in L^1$ et un ensemble mesurable $E$ arbitraire, de mesure finie et notons $E'$ le sous-ensemble majeur donn\'e par
$$ E':=\{x\in E,\ M(F)(x)>K|E|^{-1} \}
$$
pour une certaine constante num\'erique $K$.
Alors, on sait que (uniform\'ement en $\omega\in[0,1]$)
$$ \int_{E'} \left|\int_0^\infty \epsilon_t(\omega) \sqrt{t} Q_t(L) \dive(F) \, \frac{dt}{t}\right|\, dx \lesssim \|F\|_{L^1}$$
car on peut appliquer la preuve de la Proposition \ref{prop:typefaible} et le fait que \cite{BernicotFreyPetermichl} s'applique pour obtenir le contr\^ole \'epars de l'op\'erateur dual
$$ \left(\int_0^\infty \epsilon_t(\omega) \sqrt{t} Q_t(L) \dive(\cdot) \, \frac{dt}{t}\right)^{*} = \int_0^\infty \epsilon_t(\omega) \sqrt{t} \nabla Q_t(L^*)  \, \frac{dt}{t}$$
(avec des constantes uniformes en $\omega\in[0,1]$).
En int\'egrant\footnote{On note ici qu'il est tr\`es important que la construction de l'ensemble majeur $E'$ est ind\'ependant du param\`etre $\omega$.} alors par rapport au param\`etre $\omega\in[0,1]$ et par le th\'eor\`eme de Fubini, on obtient que 
$$ \int_{E'} \E_\omega\left[ \left|\int_0^\infty \epsilon_t(\omega) \sqrt{t} Q_t(L) \dive(F) \, \frac{dt}{t} \right| \right] \, dx \lesssim \|F\|_{L^1}, $$
ce qui conclut la preuve de 
$$ \left\| \left(\int_0^\infty \left| \sqrt{t} Q_t(L) \dive(F)  \right|^2 \, \frac{dt}{t} \right)^{1/2} \right\|_{L^{1,\infty}} \lesssim \|F\|_{L^1},$$
de par la caract\'erisation \eqref{eq:caracterisationL1faible} de la norme $L^{1,\infty}$.
\end{proof}

\subsection{\'Etude du projecteur de Leray}

Associ\'e \`a un tel op\'erateur $L$, on peut consid\'erer le {\it projecteur de Leray} $\pi_L$ d\'efini (formellement) par
$$ \pi_L(f) := A \nabla L^{-1} \dive.$$
Cet op\'erateur a \'et\'e introduit pour l'\'etude des \'equations de Navier-Stokes, puisque $\textrm{Id}-\pi_L$ peut \^etre vu comme le projecteur sur les champs de vecteurs \`a divergence nulle. En effet, pour un champ de vecteur $F:{\rr R}^n \rightarrow {\rr R}^n$, on a la d\'ecomposition
$$ F = G+\pi_L(F)$$
avec $G:=F-\pi_L(F)$ et la propri\'et\'e $\pi_L^2=\pi_L$. De plus, par d\'efinition, on a 
$$\dive(G) = \dive(F) - \dive(\pi_L F) =0$$ 
ce qui nous permet de d\'ecomposer un champ de vecteur $F$ comme la somme d'un champ de vecteur \`a divergence nulle $G$ et d'un champ de vecteur de type 'gradient' $\pi_L(F)$ (relatif \`a $A$).

\begin{remark} Il est plus naturel d'\'ecrire les choses en termes de m\'etrique perturb\'ee. Soit $B:=B(x)$ une application \`a valeurs matricielles v\'erifiant une condition d'ellipticit\'e et d\'efinissons $A:=B^* B$. On peut alors d\'efinir les op\'erateurs de gradient et de divergence relatifs \`a $B$: $\nabla_B:=B \nabla $ et $\dive_B:=\dive B^{*}$, de sorte que $\dive_B \nabla_B = -L_A$. Le projecteur de Leray $\pi_L:=\nabla_B L_A^{-1} \dive_B$ permet alors de d\'ecomposer un champ de vecteur, comme la somme d'un champ de vecteur \`a divergence nulle ($\dive_B$) et d'un champ de vecteur de type gradient ($\nabla_B$). \\
Un tel choix d'application $B$ correspond \`a une structure Riemannienne (\'eventuellement non r\'eguli\`ere) sur ${\rr R}^n$. En effet, il existe une \'equivalence entre les deux points de vue: perturber l'op\'erateur Laplacien par une telle application $B$ et perturber la m\'etrique Riemannienne sur l'espace Euclidien de mani\`ere quasi-isom\'etrique (pour plus de d\'etails, voir \cite{Barbatis} et \cite[Section 4]{CoulhonDungey}).
\end{remark}

\medskip

En d\'efinissant la transform\'ee de Riesz $R_L:=\nabla L^{-1/2}$, on observe que
$$ \pi_L=  A R_L \left(R_{L^*}\right)^*.$$
Donc le projecteur de Leray $\pi_L$ est la composition de deux op\'erateurs singuliers non-int\'egraux. Sous les hypoth\`eses pr\'ec\'edentes, \cite{BernicotFreyPetermichl} montre que les transform\'ees de Riesz $R_L$ et $R_{L^*}$ v\'erifient les estimations \`a poids optimales suivantes: pour tout $p\in(1,\infty)$ et tout poids $\omega \in A_p$ alors
$$ \| R_L \|_{L^p_\omega \to L^p_\omega} + \| R_{L^*} \|_{L^p_\omega \to L^p_\omega} \lesssim |\omega]_{A_p}^{\max\{1,1/(p-1)\}}.$$
Par dualit\'e (avec le poids dual $\sigma:=\omega^{1-p'}$), on en d\'eduit que
\begin{align*}
 \| \left(R_{L^*}\right)^{*} \|_{L^p_\omega \to L^p_\omega} & = \| R_{L^*} \|_{L^{p'}_{\sigma} \to L^{p'}_{\sigma}} \lesssim [\omega^{1-p'}]_{A_{p'}}^{\max\{1,1/(p'-1)\}} \\
 & \lesssim [\omega]_{A_{p}}^{(p'-1)\max\{1,1/(p'-1)\}} \lesssim [\omega]_{A_{p}}^{\max\{1,1/(p-1)\}}.
 \end{align*}
 En composant ces deux estimations, on obtient
 $$ \| \pi_L\|_{L^p_\omega \to L^p_\omega} \lesssim [\omega]_{A_{p}}^{2\max\{1,1/(p-1)\}}.$$

\medskip


La transform\'ee de Riesz est un op\'erateur qui v\'erifie $R_L(1)=0$ donc en suivant la 'philosophie' de la section \ref{sec:T1=0} (la fin sur la composition d'op\'erateurs de Calder\'on-Zygmund), on doit s'attendre \`a pouvoir faire int\'eragir cette cancellation et obtenir de meilleures estimations pour le projecteur de Leray, en tant que composition de deux op\'erateurs singuliers. La proposition suivante montre en effet que m\^eme si les transform\'ees de Riesz ne sont pas des op\'erateurs de Calder\'on-Zygmund, la compos\'ee $\pi_L$ v\'erifie bien de meilleures estimations \`a poids, de mani\`ere analogue \`a la composition de deux op\'erateurs de Calder\'on-Zygmund v\'erifiant la condition $T(1)=0$ (voir la fin de la section \ref{sec:T1=0}):

\begin{proposition} Sous les hypoth\`eses pr\'ec\'edentes, pour tout champs de vecteurs $F,G \in L^2$ alors il existe une collection \'eparse $\mcS$ telle que
$$ \left|\langle \pi_L(F),G\rangle \right| \lesssim \sum_{Q\in \mcS} \left(\aver{Q} |F| \, dx\right) \left(\aver{Q} |G| \, dx\right) |Q|.$$
En cons\'equence, pour tout $p\in(1,\infty)$ et poids $\omega \in A_{p}$, on a
$$ \| \pi_L  \|_{L^p_\omega \to L^p_\omega} \lesssim_\eta [\omega]_{A_{p}}^{\max\{1,1/(p-1)\}}.$$
et $\pi_L$ est de type faible $L^1$.
\end{proposition}

\begin{remark} En appliquant le Th\'eor\`eme \ref{thm:hardy}, on obtient aussi des estimations du projecteur de Leray dans des espaces de Hardy \`a poids.
\end{remark}

Dans le cas classique o\`u $A=\textrm{Id}$ alors $L$ est le Laplacien Euclidien et il est bien connu que le projecteur de Leray $\pi_L$ est un multiplicateur de Fourier v\'erifiant les conditions de Mikhlin et donc est en particulier un op\'erateur de Calder\'on-Zygmund. Ici, on \'etend les diff\'erentes estimations (type faible $L^1$ et estimations $L^p$ \`a poids), dans le cas o\`u $\pi_L$ n'est pas un op\'erateur de Calder\'on-Zygmund.

\begin{proof}
On utilise une formule de reproduction de Calder\'on, pour obtenir pour tout champ de vecteur $F,G\in L^2$
\begin{align*}
 \langle \pi_L(F) , g\rangle & = \int_0^\infty \langle \sqrt{t} (tL)^{N+1} e^{-tL}  \dive(F) ,  \sqrt{t} (tL^*)^N e^{-tL^*} \dive(A^* G) \rangle \, \frac{dt}{t} \\
& = \int_0^\infty \langle tA\nabla (tL)^N e^{-tL}  \dive(F) ,  t\nabla (tL^*)^N e^{-tL^*} \dive(A^* G) \rangle \, \frac{dt}{t}.
\end{align*}
En particulier pour le choix $N=0$, on obtient
 \begin{align*}
 \langle \pi_L(F) , g\rangle &  = \int_0^\infty \langle tA\nabla e^{-tL}  \dive(F) ,  t\nabla e^{-tL^*} \dive(A^* G) \rangle \, \frac{dt}{t} \\
 & = \sum_{k \in {\rr Z}} \int_{2^k}^{2^{k+1}} \langle t A \nabla e^{-tL}  \dive(F) ,  t \nabla e^{-tL^*} \dive(A^* G) \rangle \, \frac{dt}{t}.
 \end{align*}
Consid\'erons alors la grille dyadique dans ${\rr R}^n$
$$ {\mathcal D}:= \left\{2^{k}\left([0,1]^n+m\right), \ k\in {\mathbb Z},\ m\in{\mathbb Z}^n \right\},$$
ce qui permet d'\'ecrire
\begin{align*}
 |\langle \pi_L(F) , G\rangle| \lesssim \sum_{I\in {\mathcal D}} a_I(F) a_I^*(G) |I|
 \end{align*}
avec les coefficients d\'efinis pour $I= 2^k\left([0,1]^n+j\right)$ par
$$ a_I(F) := \left(\int_{2^{2k}}^{2^{2k+2}}  \|t \nabla e^{-tL}  \dive(F) \|_{L^2(I)}^2 \, \frac{dt}{t} \right)^{1/2} $$
et
$$ a_I^*(G) := \left(\int_{2^{2k}}^{2^{2k+2}}  \|t \nabla e^{-tL^*}  \dive(A^* G) \|_{L^2(I)}^2 \, \frac{dt}{t} \right)^{1/2}.$$
Il nous reste alors \`a v\'erifier la propri\'et\'e principale \eqref{pr:size} pour les coefficients $a_I(F)$ et $a_I^* (G)$. \'Etant donn\'e que les hypoth\`eses sur $L$ sont sym\'etriques et la multiplication par $A^*$ sur les champs de vecteurs laisse invariante les normes et moyennes $L^p$, il nous suffit de v\'erifier pour  $a_I(F)$.

Fixons un intervalle dyadique $I_0$ et on souhaite controler $\sup_{I_1 \subset I_0} \size(\cdot,I_1)$ o\`u
$$ \size(F,I_1):= \Big(\frac{1}{|I_1|} \sum_{I\subset I_1} |a_I(F)|^2\Big)^{1/2}.$$
Comme d\'ecrit dans \cite[Lemma 4.2]{MTTBiest0}, cette quantit\'e satisfait une in\'egalit\'e de John-Nirenberg et on a
\begin{equation} \sup_{I_1 \subset I_0} \ \size(F,I_1) \ \lesssim \ \sup_{I_1 \subset I_0} \ \sssize(F,I_1)  \label{eq:size0}
\end{equation}
avec la nouvelle quantit\'e 
$$ \sssize(F,I_1):= \frac{1}{|I_1|} \Big\| \Big(\sum_{I\subset I_1} |a_I(F)|^2   |I|^{-1} {\bf 1}_I\Big)^{1/2} \Big\|_{L^{1,\infty}(I_1)}.$$
Examinons alors la fonction carr\'ee apparaissant dans l'expression ci-dessus: pour tout intervalle $I_1 \subset I_0$ et tout point $x\in I_1$
\begin{align*}
\sum_{I\subset I_1} |a_I(F)|^2   |I|^{-1} {\bf 1}_I(x) & = \sum_{\genfrac{}{}{0pt}{}{k\in {\rr Z}}{2^k \leq |I_1|}} \sum_{\genfrac{}{}{0pt}{}{j}{j2^k \in I_1}}  
\int_{2^{2k}}^{2^{2k+2}}  \| t \nabla e^{-tL}  \dive(F) \|_{L^2(I)}^2  2^{-kn} {\bf 1}_{|x-j2^{k}|\leq 2^k}  \, \frac{dt}{t} \\
& \lesssim  \iint_{\genfrac{}{}{0pt}{}{y,t}{|x-y|\lesssim \sqrt{t}\lesssim \ell_{I_1}}} \left| t\nabla e^{-tL}  \dive(F)(y) \right|^2 \, \frac{dydt}{t^{1+n/2}}.
\end{align*}
Donc pour une certaine constante num\'erique $\kappa$, on en d\'eduit que 
$$ \Big(\sum_{I\subset I_1} |a_I(F)|^2   |I|^{-1} {\bf 1}_I \Big)^{1/2} \lesssim {\mathcal C}_{\kappa,I_1}(F)$$
o\`u ${\mathcal C}_{\kappa,I_1}$ est la fonction carr\'e conique tronqu\'ee, d\'efinie par
$$ {\mathcal C}_{\kappa,I_1} (F)(x) := \Big(\iint_{\genfrac{}{}{0pt}{}{|x-y|\leq \kappa \sqrt{t}}{\sqrt{t}\leq \kappa \ell_{I_1}}} \left| t \nabla e^{-tL}  \dive(F)(y) \right|^2 \, \frac{dydt}{t^{1+n/2}}\Big)^{1/2}.$$ 
Ici, $\ell_{I_1}$ est la longueur du cube $I_1$.
Notons aussi ${\mathcal C}_\kappa:={\mathcal C}_{\kappa,{\rr R}^n}$ la fonction conique compl\`ete. 

En utilisant le lemme \ref{Lemme:fonctioncarre}, on en d\'eduit que la fonction carr\'e conique tronqu\'ee ${\mathcal C}_{\kappa,J}$ v\'erifie l'estimation de type faible suivante:
$$  \frac{1}{|J|}\left\| {\mathcal C}_{\kappa,J}(F)\right\|_{L^{1,\infty}(J)} \lesssim_\kappa \sum_{\ell \geq 0} 2^{-\ell N} \left(\aver{2^\ell J} |F(y)| \, dy \right),$$
o\`u $N$ est un entier fix\'e (aussi grand que l'on veut).
On obtient donc l'estimation suivante: pour tout cube dyadique $I_1$,
$$ \sssize(F,I_1) \lesssim  \sum_{\ell \geq 0} 2^{-\ell N} \left(\aver{2^\ell I_1} |F(y)| \, dy \right) \lesssim \left(\frac{1}{|I_1|}\int  |F(y)| \chi_{I_1}(y)^M \, dy\right).$$
De \eqref{eq:size0}, on conclut que
$$ \sup_{I_1 \subset I_0} \ \sssize(F,I_1) \ \lesssim \ \sup_{I_1 \subset I_0}   \ \left(\frac{1}{|I_1|}\int  |F(y)| \chi_{I_1}(y)^M \, dy \right),$$
qui correspond \`a \eqref{pr:size}.
On peut donc appliquer les r\'esultats de la section pr\'ec\'edente, Th\'eor\`emes \ref{thm:p>1} et \ref{thm:hardy}.
\end{proof}

%

\begin{lemma}\label{Lemme:fonctioncarre}  La fonction carr\'ee conique tronqu\'ee ${\mathcal C}_{\kappa,J}$ v\'erifie l'estimation de type faible suivante:
$$  \frac{1}{|J|}\left\| {\mathcal C}_{\kappa,J}(F)\right\|_{L^{1,\infty}(J)} \lesssim_\kappa \sum_{\ell \geq 0} 2^{-\ell N} \left(\aver{2^\ell J} |F(y)| \, dy \right).$$
\end{lemma}

\begin{proof} On rappelle la d\'efinition de la fonction conique
$$ {\mathcal C}_{\kappa,J} (F)(x) := \Big(\iint_{\genfrac{}{}{0pt}{}{|x-y|\leq \kappa \sqrt{t}}{\sqrt{t}\leq \kappa \ell_J}} \left| t \nabla e^{-tL}  \dive(F)(y) \right|^2 \, \frac{dydt}{t^{1+n/2}}\Big)^{1/2}.$$
En suivant \cite{AuscherRuss} (ou \cite[Section 7.3]{AMcR} dans un cadre g\'eom\'etrique), on introduit aussi la fonction maximale non-tangentielle
$$ {\mathcal M}_{\kappa,J}^{max} (F)(x) := \sup_{\genfrac{}{}{0pt}{}{|x-y|\leq \kappa \sqrt{t}}{\sqrt{t}\leq \kappa \ell_J}}  \left|\sqrt{t} e^{-tL}  \dive(F)(y) \right|.$$
Alors par  \cite[Lemma 9]{AuscherRuss} (ou \cite[Section 7.3]{AMcR} dans un cadre g\'eom\'etrique), la fonction carr\'e conique satisfait une estimation aux ``bons-$\lambda$'' par rapport \`a la fonction maximale non-tangentielle. Il est bien connu qu'une telle in\'egalit\'e permet alors d'avoir l'estimation suivante:
$$ \left\| {\mathcal C}_{\kappa,J} (F) \right\|_{L^{1,\infty}(J)} \lesssim \left\| {\mathcal M}_{\kappa,J}^{max} (F) \right\|_{L^{1,\infty}(J)}.$$
En utilisant les hypoth\`eses faites sur le semi-groupe, il est maintenant facile de contr\^oler ponctuellement la fonction maximale non-tangentielle: en effet pour tout $x\in J$  et tout $(y,t)$ tel que $|x-y|\leq \sqrt{t} \leq \kappa \ell_J$ alors, on  a
$$ \left|\sqrt{t} e^{-tL}  \dive(F)(y) \right| \lesssim_\kappa {\mathcal M} \left[\left(1+\frac{d(\cdot,J)}{r(J)}\right)^{-N} F \right](x) $$
o\`u ${\mathcal M}$ est la fonction maximale de Hardy-Littlewood. En utilisant le type $L^1$-faible de ${\mathcal M}$, on en d\'eduit donc
\begin{align*}
 \frac{1}{|J|} \left\| {\mathcal M}_{\kappa,J}^{max} (F) \right\|_{L^{1,\infty}(J)} 
 & \lesssim \frac{1}{|J|} \left\| {\mathcal M} \left[\left(1+\frac{d(\cdot,J)}{r(J)}\right)^{-N} F \right] \right\|_{L^{1,\infty}(J)} \\
 & \lesssim  \sum_{\ell \geq 0} 2^{-\ell N} \left(\aver{2^\ell J} |F(y)| \, dy \right)
 \end{align*}
o\`u $N$ est un entier aussi grand que l'on souhaite (pouvant changer d'une ligne \`a une autre).
\end{proof}

\section{Espace $BMO$, mesures de Carleson}
\label{sec:arbres}
Dans \cite{TreesBMOCarlesonMeasures}, les auteurs pr\'esentent le lien entre les `size', comme ils sont utilis\'es dans l'analyse de temps-fr\'equence, l'espace $BMO$, et les mesures de Carleson.  Les mesures de Carleson sont apparues pour la premi\`ere fois dans la d\'ecomposition en couronnes de \cite{carleson-corona}, et elles ont jou\'e un r\^ole important aussi dans le probl\`eme de racine carr\'ee de Kato.

Le `size' est d'habitude un op\'erateur maximal d\'efini sur des arbres (structure dans le plan de temps-fr\'equence). Mais si l'information fr\'equentielle depend seulement de la longueur de l'intervalle spatial (dans ce cas-l\`a on dit qu'on a une collection de rang $0$), le 'size` devient un sup apr\`es des collections d'intervalles, comme d\'efini au d\'ebut de section \ref{sec:haar}. En utilisant des ondelettes de Haar, on peut voir que la norme $BMO$ et le 'size` sont vraiment la m\^eme chose... Dans ce m\^eme cadre, nous allons montrer deux autres propri\'et\'es: l'existence d'une d\'ecomposition atomique \'eparse ainsi que la domination \'eparse par des oscillations.

\subsection{Une d\'ecomposition atomique \'eparse pour des fonctions de l'espace de Hardy}
\label{sec:atomic-dec}

Dans cette section, on montre l'existence d'une d\'ecomposition atomique \'eparse, dans le cadre sp\'ecial des espaces de Hardy associ\'es aux fonctions des Haar. 
On consid\`ere donc $(h_I)$ le syst\`eme de fonctions de Haar (normalis\'ees dans $L^2$) adapt\'e \`a la collection dyadique $\Dy^0$. Pour un exposant $p\in(0,1]$ l'espace de Hardy $H^p$ est associ\'e \`a la quasi-norme $L^p$ de la fonction carr\'ee
\begin{equation} S(f):= \Big( \sum_{I\in \Dy^0} |\langle f, h_I\rangle|^2 \frac{1_I}{|I|} \Big)^{1/2}. \label{eq:fonctioncarree}
\end{equation}

Nous allons montrer le r\'esultat suivant:

\begin{proposition} Soit $f$ une fonction appartenant \`a l'espace de Hardy $H^p$, qu'on peut \'ecrire comme une combinaison lin\'eaire finie des fonctions de Haar:
\[
f(x)=\sum_{I \in \ii{I} } \langle f, h_I \rangle h_I(x),
\]
o\`u $\ii I$ est une collection finie des intervalles dyadiques.\footnote{On rappelle que ce sous-espace de $H^p$ est dense dans $H^p$.}
Alors, il existe une collection $\mcS$ d'intervalles, une collection de coefficients $(c_Q)_{Q\in \mcS}$ et des atomes $(a_Q)_Q$ associ\'es \`a $\mcS$ tels que
\begin{itemize}
\item[$\bullet$] On a une d\'ecomposition atomique de 
$$ f =\sum_{Q \in \ic{S}} c_Q a_Q$$
\item[$\bullet$] $\mcS$ est une collection \'eparse
\item[$\bullet$] la d\'ecomposition atomique r\'ealise la norme $H^p$, c'est \`a dire que 
$$ \sum_{Q \in \mathcal{S}} c_Q^p \lesssim \|f\|_{H^p}^p.$$
\end{itemize}
\end{proposition}

On rappelle tout d'abord qu'un atome $a_Q$ est une fonction de $L^2(Q)$, telle que $\int_Q a_Q dx=0$, et avec la normalisation
$$ \| a_Q\|_2 \leq \vert Q \vert^{\frac{1}{2}-\frac{1}{p}}.$$

La nouveaut\'e de ce r\'esultat est d'obtenir une d\'ecomposition atomique qui est simultanément support\'ee sur une collection \'eparse. La d\'ecomposition atomique pr\'esent\'ee en \cite{TreesBMOCarlesonMeasures} a des propri\'et\'es similaires, mais les supports des atomes ne forment pas une collection \'eparse.

\begin{proof}
Pour simplicit\'e, on suppose que $\| Sf  \|_p=1 $ et on veut \'ecrire $f$ comme
\[
f(x)=\sum_{Q \in \ic{S}} c_Q a_Q(x), \quad \text{o\`u} \quad \| a_Q\|_2 \leq \vert Q \vert^{\frac{1}{2}-\frac{1}{p}} \text{   et   } \sum_{Q \in \mathcal{S}} c_Q^p \lesssim 1.
\]
On choisi $r$ un nombre r\'eel telle que $0<r<p$, et on veut construire (d'une mani\`ere it\'erative) une collection \'eparse $\ic S$ d'intervalles dyadique satisfaisant les propri\'et\'es suivantes:
\begin{enumerate}
\item[(i)] $\ic S =\bigcup_{k \geq 0} \ic S_k$.
\item[(ii)] Pour chaque $Q_0 \in \ic S_k$, il existe une collection $\ii I_{Q_0} \subseteq \ii I$ associ\'ee \`a l'intervalle $Q_0$, ainsi qu'une collection $ch_{\ic S}(Q_0)$ form\'ee par des descendants directs de $Q_0$.
\item[(iii)] D'ailleurs, $\ic S_{k+1}:=\bigcup_{Q_0 \in \ic S_k} ch_{\ic S}(Q_0)$.
\item[(iv)] La condition \'eparse est \'equivalente \`a $\ds \sum_{Q \in ch_{\ic S}(Q_0)} \vert Q \vert \leq \frac{1}{2} \vert Q_0 \vert$.
\end{enumerate}

Comme avant, on initialise $\ii I_{Stock}:= \ii I$, et on d\'efini $\ic S_0$ comme la collection des intervalles maximaux: 
$$\ic S_0:=\lbrace Q_0: Q_0 \in \ii I \text{  intervalle maximal}  \rbrace.$$
Pour chaque $Q_0 \in \ic S_0$, la collection $\ii I_{Q_0}$ se constitue par les intervalles $I' \in \ii I_{Stock}$ avec la propri\'et\'e
\[
\frac{1}{\vert I' \vert^{1/r}} \big\| \big(  \sum_{\substack{I \in \ii I_{Stock} \\ I \subseteq I'}}  \vert \langle f, h_I  \rangle \vert^2 \frac{\one_I}{\vert I \vert} \big)^{1/2} \big\|_r \leq C \cdot \frac{1}{\vert Q_0 \vert^{1/r}} \big\| \big(  \sum_{\substack{I \in \ii I_{Stock} \\ I \subseteq Q_0}} \vert \langle f, h_I  \rangle \vert^2 \frac{\one_I}{\vert I \vert} \big)^{1/2} \big\|_r.
\]
De l'autre c\^ot\'e, la collection des descendants $ch_{\ic S}(Q_0)$ se compose des intervalles dyadiques maximaux $Q \in \ii I_{Stock}$, avec $Q \subseteq Q_0$, pour lesquels 
\[
\frac{1}{\vert Q \vert^{1/r}} \big\| \big(  \sum_{\substack{I \in \ii I_{Stock} \\ I \subseteq Q}}  \vert \langle f, h_I  \rangle \vert^2 \frac{\one_I}{\vert I \vert} \big)^{1/2} \big\|_r > C \cdot \frac{1}{\vert Q_0 \vert^{1/r}} \big\| \big(  \sum_{\substack{I \in \ii I_{Stock} \\ I \subseteq Q_0}} \vert \langle f, h_I  \rangle \vert^2 \frac{\one_I}{\vert I \vert} \big)^{1/2} \big\|_r.
\]
Ainsi les intervalles des $ch_{\ic S}(Q_0)$ sont deux \`a deux disjoints et ils sont contenus dans l'ensemble 
\[
 \Big\lbrace x: {\ic M}_r \left( S_{Q_0} f \right)> C \frac{1}{\vert  Q_0 \vert ^{1/r}} \|  \big(  \sum_{\substack{I \in \ii I_{Stock} \\ I \subseteq Q_0}} \vert \langle f, h_I  \rangle \vert^2 \frac{\one_I}{\vert I \vert} \big)^{1/2} \|_r \Big\rbrace.
\]
Cela permet de montrer la condition \'eparse:
\begin{align*}
\sum_{Q \in ch_{\ic S}\left( Q_0 \right)} \vert Q \vert &\leq \big\vert \big\lbrace x: {\ic M}_r \big( \big(  \sum_{\substack{I \in \ii I_{Stock} \\ I \subseteq Q_0}} \vert \langle f, h_I  \rangle \vert^2 \frac{\one_I}{\vert I \vert} \big)^{1/2}\big)> C \frac{1}{\vert  Q_0 \vert ^{1/r}} \|  \big(  \sum_{\substack{I \in \ii I_{Stock} \\ I \subseteq Q_0}} \vert \langle f, h_I  \rangle \vert^2 \frac{\one_I}{\vert I \vert} \big)^{1/2} \|_r \big\rbrace  \big\vert \\
&\lesssim \big( C \frac{1}{\vert  Q_0 \vert ^{1/r}} \|  \big(  \sum_{\substack{I \in \ii I_{Stock} \\ I \subseteq Q_0}} \vert \langle f, h_I  \rangle \vert^2 \frac{\one_I}{\vert I \vert} \big)^{1/2} \|_r  \big)^{-r} \big\|  {\ic M}_r \big( \big(  \sum_{\substack{I \in \ii I_{Stock} \\ I \subseteq Q_0}} \vert \langle f, h_I  \rangle \vert^2 \frac{\one_I}{\vert I \vert} \big)^{1/2}\big) \big\|_{r,\infty}^r \\
&\lesssim C^{-r} \vert Q_0  \vert,
\end{align*}
qui est born\'e par $\frac{1}{2}\vert Q_0 \vert$ pour une valeur de $C$ assez grande. 

Apr\`es avoir d\'efini $\ic S_0$ et, pour chaque $Q_0 \in \ic S_0$, les collections associ\'ees $\ii I_{Q_0}$ et $ch_{\ic S}(Q_0)$, on r\'einitialise $\ii I_{Stock}:= \ii I_{Stock}\setminus \bigcup_{Q_0 \in \ic S_0} \ii I_Q$. 

La proc\'edure continue par it\'eration, mais seulement pour un nombre fini des pas, car la collection initiale des intervalles $\ii I$ est finie.


\`A la fin, on a une collection \'eparse $\ic S$ et une d\'ecomposition de la fonction $f$:
\[
f(x):=\sum_{Q \in \ic S} \sum_{I \in \ii I_Q} \langle f, h_I \rangle h_I(x)= \sum_{Q \in \ic S} c_Q a_Q(x),
\]
o\`u 
$$\ds c_Q:=\vert Q \vert^{\frac{1}{p}} \cdot \frac{1}{\vert Q \vert^{\frac{1}{2}}} \big\| \big(  \sum_{ I \in \ii I_Q}  \vert \langle f, h_I  \rangle \vert^2 \frac{\one_I}{\vert I \vert} \big)^{1/2}  \big\|_2 \quad \textrm{et} \quad a_Q:=c_Q^{-1} \sum_{I \in \ii I_Q} \langle f, h_I \rangle h_I(x).$$

Il n'est pas difficile de v\'erifier que chaque $a_Q$ est un atome dans $H^p$, dans le sens de \cite{TreesBMOCarlesonMeasures} (la propri\'et\'e de cancellation de l'atome $a_Q$ est impliqu\'ee par le fait que l'atome est engendr\'e par des fonctions de Haar $h_I$ avec $I\subset Q$). Il reste \`a v\'erifier que $\sum_{Q \in \ic S}c_Q^p \lesssim 1$. On remarque que 
\begin{align*}
&\sum_{Q \in \ic S} c_Q^p =\sum_{Q \in \ic S} \Big( \frac{1}{\vert Q \vert^{\frac{1}{2}}} \big\| \big(  \sum_{ I \in \ii I_Q}  \vert \langle f, h_I  \rangle \vert^2 \frac{\one_I}{\vert I \vert} \big)^{1/2}  \big\|_2  \Big)^p \cdot \vert Q\vert \\
& \lesssim \sum_{Q \in \ic S } \Big( \sup_{ \tilde Q \in \ii I_Q} \frac{1}{\vert \tilde  Q \vert^{\frac{1}{2}}} \big\| \big(  \sum_{\substack{ I \in \ii I_Q \\ I \subseteq \tilde Q}}  \vert \langle f, h_I  \rangle \vert^2 \frac{\one_I}{\vert I \vert} \big)^{1/2}  \big\|_2  \Big)^p \cdot \vert Q \vert,
\end{align*}
car $\ii I_Q$ contient $Q$.
Gr\^ace \`a l'in\'egalit\'e de John-Nirenberg, les moyennes $L^2$ et $L^r$ de l'oscillation sont comparables:
\[
\sup_{ \tilde Q \in \ii I_Q} \frac{1}{\vert \tilde  Q \vert^{\frac{1}{2}}} \big\| \big(  \sum_{\substack{ I \in \ii I_Q \\ I \subseteq \tilde Q}}  \vert \langle f, h_I  \rangle \vert^2 \frac{\one_I}{\vert I \vert} \big)^{1/2}  \big\|_2 \sim _{r} \sup_{ \tilde Q \in \ii I_Q} \frac{1}{\vert \tilde  Q \vert^{\frac{1}{r}}} \big\| \big(  \sum_{\substack{ I \in \ii I_Q \\ I \subseteq \tilde Q}}  \vert \langle f, h_I  \rangle \vert^2 \frac{\one_I}{\vert I \vert} \big)^{1/2}  \big\|_r.
\]
Maintenant on utilise les propri\'et\'es des collections construites par le temps d'arr\^et: chaque intervalle $I' \in \ii I_Q$ est tel que
\[
\frac{1}{\vert I' \vert^{\frac{1}{r}}} \big\| \big(  \sum_{\substack{ I \in \ii I_{Stock}^* \\ I \subseteq I'}}  \vert \langle f, h_I  \rangle \vert^2 \frac{\one_I}{\vert I \vert} \big)^{1/2}  \big\|_r \leq C \frac{1}{\vert   Q \vert^{\frac{1}{r}}} \big\| \big(  \sum_{\substack{ I \in \ii I_{Stock}^* \\ I \subseteq  Q}}  \vert \langle f, h_I  \rangle \vert^2 \frac{\one_I}{\vert I \vert} \big)^{1/2}  \big\|_r,
\]
o\`u par $\in \ii I_{Stock}^*$ on d\'enote la collection des intervalles disponibles au moment du temps d'arr\^et o\`u $\ii I_Q$ a \'et\'e construit. En particulier, $\ii I_Q$ est contenue dans $\in \ii I_{Stock}^*$ et en cons\'equence, on a
\[
\sup_{ \tilde Q \in \ii I_Q} \frac{1}{\vert \tilde  Q \vert^{\frac{1}{r}}} \big\| \big(  \sum_{\substack{ I \in \ii I_Q \\ I \subseteq \tilde Q}}  \vert \langle f, h_I  \rangle \vert^2 \frac{\one_I}{\vert I \vert} \big)^{1/2}  \big\|_r \leq C  \frac{1}{\vert   Q \vert^{\frac{1}{r}}} \big\| \big(  \sum_{\substack{ I \in \Dy^0 \\ I \subseteq  Q}}  \vert \langle f, h_I  \rangle \vert^2 \frac{\one_I}{\vert I \vert} \big)^{1/2}  \big\|_r.
\]
Dans l'expression ci-dessus, $\Dy^0$ d\'enote la collection de tous intervalles dyadiques, et on rappelle la notation $Sf(x)$ pour la fonction carr\'ee ``totale''  associ\'ee:
$$ Sf(x):=\big(  \sum_{I \in \Dy^0} \vert \langle f, h_I  \rangle \vert^2 \frac{\one_I}{\vert I \vert} \big)^{1/2}.$$

D'autre part, le caract\`ere \'epars de la collection $\ic S$ implique, pour tout $Q \in \ic S$, l'existence d'un sous-ensemble majeur $E(Q) \subseteq Q$ telle que les ensembles  $\lbrace E(Q) \rbrace_{Q \in \ic S}$ sont deux par deux disjoints. En observant que 
\[
\frac{1}{\vert Q \vert^{\frac{1}{r}}} \big\| \big(  \sum_{\substack{ I \in \ii I_Q \\ I \subseteq Q}}  \vert \langle f, h_I  \rangle \vert^2 \frac{\one_I}{\vert I \vert} \big)^{1/2}  \big\|_r \lesssim \inf_{y \in E\left( Q \right)} \ic M_r (Sf)(y),
\]
on obtient l'estimation 
\begin{align*}
\sum_{Q \in \ic S} c_Q^p &\lesssim \sum_{Q \in \ic S} \left( \inf_{y \in E\left( Q \right)} \ic M_r (Sf)(y) \right)^p \cdot 2 \vert E(Q)  \vert \\
&\lesssim \int_{\rr R} \vert \ic M_r (Sf)(x) \vert^p dx \lesssim \int_{\rr R} \vert Sf (x) \vert^p dx \lesssim 1.
\end{align*}
Ci-dessus on utilise la bornitude $L^p$ de la fonction maximale $\ic M_r$, qui est une cons\'equence de la condition $r<p$.
\end{proof}

\begin{remark}
On remarque que les atomes contiennent l'information des paquets d'onde de la collection $\ii I_Q$, qui a \'et\'e choisie telle que 
{\fontsize{10.5}{10}\[
\frac{1}{\vert Q \vert^{\frac{1}{2}}} \big\| \big(  \sum_{\substack{ I \in \ii I_Q \\ I \subseteq Q}}  \vert \langle f, h_I  \rangle \vert^2 \frac{\one_I}{\vert I \vert} \big)^{1/2}  \big\|_2 \sim \frac{1}{\vert Q \vert^{\frac{1}{r}}} \big\| \big(  \sum_{\substack{ I \in \ii I_Q \\ I \subseteq Q}}  \vert \langle f, h_I  \rangle \vert^2 \frac{\one_I}{\vert I \vert} \big)^{1/2}  \big\|_r \sim  \big\| \big(  \sum_{\substack{ I \in \ii I_Q \\ I \subseteq Q}}  \vert \langle f, h_I  \rangle \vert^2 \frac{\one_I}{\vert I \vert} \big)^{1/2}  \big\|_{BMO}.
\]}
\end{remark}

\subsection{Une certain regard sur la domination \'eparse par des oscillations} \label{sec:oscillation}

On reste toujours dans le cadre de fonctions de Haar, et on consid\`ere un multiplicateur de Haar associe \`a une collection finie des intervalles $\ii I \subset \Dy^0$:
\[
Tf(x):=\sum_{I \in \ii I} \epsilon_I \langle f, h_I \rangle h_I(x), 
\]
et sa forme bilin\'eaire $\ds \Lambda_{\ii I}(f, g):=\sum_{I \in \ii I} \epsilon_I \langle f, h_I \rangle \langle g, h_I \rangle$. Ici les coefficients $a_I$ sont uniform\'ement bornes: $\vert \epsilon_I \vert \lesssim 1$ pour tous $I \in \ii I$.

\begin{proposition}
\label{prop-domination-oscillations}
Si $T$ est un multiplicateur de Haar comme d\'efini ci-dessus, et si $f$ et $g$ sont deux fonctions dans $L^2$, alors il existe une collection \'eparse $\ic S$, d\'ependant de $f$ et $g$, telle que
\[
\vert \Lambda_{\ii I}(f, g) \vert \lesssim \sum_{Q \in \ic S} \left(\frac{1}{\vert Q \vert} \int_Q \vert f(x) -\aver{Q}f  \vert dx\right) \cdot \left(\frac{1}{\vert Q \vert} \int_Q \vert g(x) -\aver{Q}g  \vert dx\right) \cdot \vert Q \vert.
\]
\end{proposition}

\begin{remark}
Dans le cas particulier o\`u l'op\'erateur co\"incide \`a l'identit\'e, on obtient pour deux fonctions $f,g\in L^2$
\[
\int_{\rr R} f(x) \cdot g(x) dx \lesssim \int_{\rr R} \ic M^{\#}f(x) \cdot \ic M^{\#}g(x) dx.
\]
Ceci am\'eliore l'in\'egalit\'e de dualit\'e \cite[(16) page 146]{SteinBigBook} et correspond \`a une version ``polaris\'ee'' de l'in\'egalit\'e $L^2$ de Fefferman-Stein \cite[Corollary 1, page 154]{SteinBigBook}.
\end{remark}

\begin{proof}
On commence par deux observations:
\begin{enumerate}[label=(\roman*), ref=\roman*]
\item \label{obs-one} si $F_Q(x):= \left( f(x) -\aver{Q} f \right) \cdot \one_Q(x)$, alors
\[
\big(  \sum_{I \subset Q } \vert \langle f, h_I \rangle  \vert^2 \cdot \frac{\one_I(x)}{\vert  I \vert}   \big)^{1/2}= \big(  \sum_{I \subset Q } \vert \langle F_Q, h_I \rangle  \vert^2 \cdot \frac{\one_I(x)}{\vert  I \vert}   \big)^{1/2}.
\]
\item \label{obs-two} Si $f$ admet une repr\'esentation dans la base de Haar $\ds f=\sum_{I \in \Dy^0} \langle f, h_I  \rangle h_I$, alors pour tout intervalle dyadique $Q$, on a
\[
\sum_{\substack{I \in \Dy^0 \\  I \subset Q }} \vert \langle f, h_I  \rangle \vert^2= \big\|  \big(  \sum_{\substack{I \in \Dy^0 \\  I \subset Q }} \vert \langle f, h_I \rangle  \vert^2 \cdot \frac{\one_I(x)}{\vert  I \vert}   \big)^{1/2} \big\|_2^2 = \int_Q \vert  f(x) -\aver{Q} f \vert^2 dx.   
\]
\end{enumerate}

La preuve de la Proposition \ref{prop-domination-oscillations} est toujours fond\'ee sur un temps d'arr\^et, mais avec des conditions un peu diff\'erentes. On initialise  $\ii I_{Stock}:=\ii I$ et $$\ic S_0:=\lbrace Q_0 \in \ii I_{Stock} : Q_0 \text{ maximal par rapport \`a l'inclusion}  \rbrace.$$

Pour tout $Q_0 \in \ic S_0$, la collection $\ii I_{Q_0}$ est form\'ee d'intervalles dyadiques $I' \in \ii I_{Stock}$ contenus dans $Q_0$, telle que
\begin{align*}
&\frac{1}{\vert I' \vert} \big\|  \big(  \sum_{I \subset I' } \vert \langle f, h_I \rangle  \vert^2 \cdot \frac{\one_I(x)}{\vert  I \vert}   \big)^{1/2} \big\|_{1, \infty} \leq C \cdot \frac{1}{\vert Q_0 \vert} \int_{Q_0} \vert  f(x) -\aver{Q_0} f \vert dx \quad \text{ et }\\
&\frac{1}{\vert I' \vert} \big\|  \big(  \sum_{I \subset I' } \vert \langle g, h_I \rangle  \vert^2 \cdot \frac{\one_I(x)}{\vert  I \vert}   \big)^{1/2} \big\|_{1, \infty} \leq C \cdot \frac{1}{\vert Q_0 \vert} \int_{Q_0} \vert  g(x) -\aver{Q_0} g \vert dx
\end{align*}

On note que $Q_0 \in \ii I_{Q_0}$: par l'observation \eqref{obs-one}, l'expression \`a gauche est $\|  S F_{Q_0}  \|_{1, \infty}$, qui est contrôl\'ee par $\| F_{Q_0}  \|_1$. La m\^eme observation est vraie pour la fonction $g$. On rappelle que $S$ est la fonction carr\'ee compl\`ete introduit en \eqref{eq:fonctioncarree}.

Les descendants de $Q_0$ sont les sous-intervalles maximaux pour lesquels la condition du temps d'arr\^et n'est plus vraie. C'est \`a dire, $ch_{\ic S}(Q_0)$ est form\'e par des intervalles de $\ii I_{Stock}$, contenus dans $Q_0$, tels que 
\begin{align*}
& \frac{1}{\vert Q \vert} \big\|  \big(  \sum_{\substack{I \subset Q    \\ I \in \ii I_{Stock}} } \vert \langle f, h_I \rangle  \vert^2 \cdot \frac{\one_I(x)}{\vert  I \vert}   \big)^{1/2} \big\|_{1, \infty} > C \cdot \frac{1}{\vert Q_0 \vert} \int_{Q_0} \vert  f(x) -\aver{Q_0} f \vert dx \quad \text{ou}\\
&\frac{1}{\vert Q \vert} \big\|  \big(  \sum_{\substack{I \subset Q    \\ I \in \ii I_{Stock}}} \vert \langle g, h_I \rangle  \vert^2 \cdot \frac{\one_I(x)}{\vert  I \vert}   \big)^{1/2} \big\|_{1, \infty} > C \cdot \frac{1}{\vert Q_0 \vert} \int_{Q_0} \vert  g(x) -\aver{Q_0} g \vert dx.
\end{align*}
Puis on r\'einitialise $\ds\ii I_{Stock}:= \ii I_{Stock} \setminus \bigcup_{Q_0 \in \ic S} \ii I_{Q_0}$ et on red\'emarre la proc\'edure en choisissant d'une mani\`ere similaire les collections $\ii I_Q$ et $ch_{\ic S}(Q)$ pour tout $Q \in ch_{\ic S}(Q_0)$. \`A la fin on obtient $\ic S :=\cup_{k\geq 0} \ic S_k$.

Il reste \`a montrer la propri\'et\'e \'eparse de la collection $\ic S$, i.e. pour chaque $Q_0 \in \ic S$, 
\[
\sum_{Q \in ch_{\ic S \left( Q_0 \right)}} \vert Q \vert \leq \frac{1}{2} \vert Q_0 \vert.
\]
La d\'efinition de la collection $ch_{\ic S}(Q_0)$ garantit que les intervalles $Q \in ch_{\ic S}(Q_0)$ sont deux \`a deux disjoints et qu'ils sont contenus dans l'ensemble
\begin{equation}
\label{eq:level-set-weak-max-f}
\Big\lbrace \ic M_{\mathfrak{w}} \big(  \big(  \sum_{I \subset Q_0 } \vert \langle F_{Q_0}, h_I \rangle  \vert^2 \cdot \frac{\one_I}{\vert  I \vert}   \big)^{1/2}  \big)(x)>  C \cdot \frac{1}{\vert Q_0 \vert} \int_{Q_0} \vert  f(x) -\aver{Q_0} f \vert dx   \Big\rbrace,
\end{equation}
o\`u, bien s\^ur, dans l'ensemble d\'efinit similairement par $g$. Ici l'op\'erateur $\ic M_{\mathfrak{w}}$ est la fonction maximale faible d\'efinie par
\[
\ic M_{\mathfrak{w}}(f)(x):=\sup_{B \ni x} \frac{1}{\vert B  \vert} \big\| f \cdot \one_B   \big\|_{1, \infty},
\]
o\`u on consid\`ere le $\sup$ apr\`es tout boules contenant $x$. La fonction maximale faible $\ic M_{\mathfrak w}$ n'agit pas de $L^{1, \infty}$ dans $L^{1, \infty}$ (\cite{weak-max-func}), mais sa composition avec un op\'erateur de Calder\'on-Zygmund agit de $L^{1}$ dans $L^{1,\infty}$ (\cite{weak-max-func+CZop}).

Par cons\'equent, la mesure de l'ensemble de \eqref{eq:level-set-weak-max-f} est major\'ee par
\[
\left(  C \cdot \frac{1}{\vert Q_0  \vert} \| F_{Q_0}  \|_1 \right)^{-1} \big\|   M_{\mathfrak{w}} \big( \sum_{I \subset Q_0 } \vert \langle F_{Q_0}, h_I \rangle  \vert^2 \cdot \frac{\one_I}{\vert  I \vert}   \big)^{1/2} \big\|_{1, \infty} \lesssim C^{-1} \vert  Q_0 \vert,
\]
qui entraine la propri\'et\'e \'eparse, en choisissant une constante $C$ assez grande.

\`A la fin, pour tout $Q \in \ic S$, on a une sous-collection $\ii I_Q$ des intervalles dyadiques et $\ds \ii I :=\bigcup_{Q \in \ic S} \ii I_Q$. Ainsi
\[
\Lambda_{\ii I}(f, g)=\sum_{Q \in \ic S} \Lambda_{\ii I_Q}(f, g).
\]

L'estimation locale pour la forme bilin\'eaire, ainsi que l'in\'egalit\'e de John-Nirenberg, impliqu\'ee par Proposition \ref{prop:localization}
{\fontsize{10.5}{10}
\begin{align*}
\vert \Lambda_{\ii I_{Q}}(f,g) \vert &\leq \sup_{I \in \ii I_Q} \vert \epsilon_I \vert \sup_{I' \in \ii I_{Q}} \frac{1}{\vert I' \vert} \big\|  \Big( \sum_{\substack{I \subseteq I' \\ I \in \ii I_{Q}}}  \frac{\vert \langle f, \phi_I \rangle \vert^2}{\vert I \vert}  \one_I \Big)^{1/2}   \big\|_{1, \infty} \cdot  \sup_{I' \in \ii I_{Q}} \frac{1}{\vert I' \vert} \big\|  \Big( \sum_{\substack{I \subseteq I' \\ I \in \ii I_{Q}}}  \frac{\vert \langle g, \phi_I \rangle \vert^2}{\vert I \vert}  \one_I \Big)^{1/2} \big\|_{1, \infty} \cdot \vert Q  \vert 
\\
&\lesssim \sup_{I \in \ii I_Q} \vert \epsilon_I \vert \cdot \left(\frac{1}{\vert Q \vert} \int_{Q} \vert  f(x) -\aver{Q} f \vert dx \right) \cdot \left(\int_{Q} \vert  g(x) -\aver{Q} g \vert dx\right),
\end{align*}}
ce qui conclut l'estimation \'eparse.
\end{proof}

\bibliographystyle{alpha}
\input{Proprietes_eparses.bbl}

\end{document}

%% file: Proprietes_eparses.bbl
\newcommand{\etalchar}[1]{$^{#1}$}

%% file: Proprietes_eparses_vf.bbl
\begin{thebibliography}{CACDPO16}

\bibitem[AHM{\etalchar{+}}02]{TreesBMOCarlesonMeasures}
P.~Auscher, S.~Hofmann, C.~Muscalu, T.~Tao, and C.~Thiele.
\newblock Carleson measures, trees, extrapolation, and {$T(b)$} theorems.
\newblock {\em Publ. Mat.}, 46(2):257--325, 2002.

\bibitem[AMR08]{AMcR}
Pascal Auscher, Alan McIntosh, and Emmanuel Russ.
\newblock Hardy spaces of differential forms on {R}iemannian manifolds.
\newblock {\em J. Geom. Anal.}, 18(1):192--248, 2008.

\bibitem[AR03]{AuscherRuss}
Pascal Auscher and Emmanuel Russ.
\newblock Hardy spaces and divergence operators on strongly {L}ipschitz domains
  of {$\Bbb R^n$}.
\newblock {\em J. Funct. Anal.}, 201(1):148--184, 2003.

\bibitem[Aus07]{Auscher-memoire}
Pascal Auscher.
\newblock On necessary and sufficient conditions for {$L^p$}-estimates of
  {R}iesz transforms associated to elliptic operators on {$\Bbb R^n$} and
  related estimates.
\newblock {\em Mem. Amer. Math. Soc.}, 186(871):xviii+75, 2007.

\bibitem[Aus08]{Auscherannexe}
Pascal Auscher.
\newblock On the {C}alder\'on-{Z}ygmund lemma for {S}obolev functions.
\newblock \url{http://arxiv.org/abs/0810.5029}, Octobre 2008.
\newblock online.

\bibitem[Bar96]{Barbatis}
Gerassimos Barbatis.
\newblock Stability of weighted {L}aplace-{B}eltrami operators under
  {$L^p$}-perturbation of the {R}iemannian metric.
\newblock {\em J. Anal. Math.}, 68:253--276, 1996.

\bibitem[BBT16]{BeneaBernicotLuque}
Cristina Benea, Fr\'ed\'eric Bernicot, and Luque Teresa.
\newblock Sparse bilinear forms for {B}ochner {R}iesz multipliers and
  applications.
\newblock \url{http://arxiv.org/abs/1605.06401}, Mai 2016.
\newblock online.

\bibitem[BFP16]{BernicotFreyPetermichl}
Fr{\'e}d{\'e}ric Bernicot, Dorothee Frey, and Stefanie Petermichl.
\newblock Sharp weighted norm estimates beyond {C}alder\'on-{Z}ygmund theory.
\newblock {\em Anal. PDE}, 9(5):1079--1113, 2016.

\bibitem[BK87]{weak-max-func+CZop}
Joaquim Bruna and Boris Korenblum.
\newblock A note on {C}alder\'on-{Z}ygmund singular integral convolution
  operators.
\newblock {\em Bull. Amer. Math. Soc. (N.S.)}, 16(2):271--273, 1987.

\bibitem[BM16]{vv_BHT}
Cristina Benea and Camil Muscalu.
\newblock Multiple vector-valued inequalities via the helicoidal method.
\newblock {\em Anal. PDE}, 9(8):1931--1988, 2016.

\bibitem[CACDPO16]{CCDO}
Jos\'e~M. Cond\'e-Alonso, Amalia Culiuc, Francesco Di~Plinio, and Yumeng Ou.
\newblock A sparse domination principle for rough singular integrals.
\newblock \url{http://arxiv.org/abs/1612.09201}, Decembre 2016.
\newblock online.

\bibitem[Car62]{carleson-corona}
Lennart Carleson.
\newblock Interpolations by bounded analytic functions and the corona problem.
\newblock {\em Ann. of Math. (2)}, 76:547--559, 1962.

\bibitem[CD07]{CoulhonDungey}
Thierry Coulhon and Nicholas Dungey.
\newblock Riesz transform and perturbation.
\newblock {\em J. Geom. Anal.}, 17(2):213--226, 2007.

\bibitem[CDPO16]{weighted_BHT}
Amalia Culiuc, Francesco Di~Plinio, and Yumeng Ou.
\newblock Domination of multilinear singular integrals by positive sparse
  forms.
\newblock \url{http://arxiv.org/abs/1603.05317}, 2016.
\newblock Online; accessed June 2016.

\bibitem[CM97]{CoifMeyer-ondelettes}
Ronald~R. Coifman and Yves Meyer.
\newblock {\em Wavelets, {Calder{\'o}n}-{Zygmund} Operators and Multilinear
  Operators}.
\newblock Cambridge University Press, 1997.

\bibitem[CUMP12]{CUMP}
David Cruz-Uribe, Jos\'e~Mar\'ia Martell, and Carlos P\'erez.
\newblock Sharp weighted estimates for classical operators.
\newblock {\em Adv. Math.}, 229(1):408--441, 2012.

\bibitem[Lac15]{LaceyA_2}
Michael Lacey.
\newblock {An elementary proof of the $A_2$ Bound}.
\newblock 14 pages, to appear in {\it Israel J. Math.}, 2015.

\bibitem[Ler10]{Lerner1}
Andrei~K. Lerner.
\newblock A pointwise estimate for the local sharp maximal function with
  applications to singular integrals.
\newblock {\em Bull. Lond. Math. Soc.}, 42(5):843--856, 2010.

\bibitem[Ler11]{Lernerfirst}
Andrei~K. Lerner.
\newblock Sharp weighted norm inequalities for {L}ittlewood-{P}aley operators
  and singular integrals.
\newblock {\em Adv. Math.}, 226(5):3912--3926, 2011.

\bibitem[Ler13a]{Lerner2}
Andrei~K. Lerner.
\newblock On an estimate of {C}alder{\'o}n-{Z}ygmund operators by dyadic
  positive operators.
\newblock {\em J. Anal. Math.}, 121:141--161, 2013.

\bibitem[Ler13b]{Lerner-simplerA_2}
Andrei~K. Lerner.
\newblock A simple proof of the {$A_2$} conjecture.
\newblock {\em Int. Math. Res. Not. IMRN}, 14:3159--3170, 2013.

\bibitem[Ler16]{Lerner4}
Andrei~K. Lerner.
\newblock On pointwise estimates involving sparse operators.
\newblock {\em New York J. Math.}, 22:341--349, 2016.

\bibitem[LL02]{LeeLin}
Ming-Yi Lee and Chin-Cheng Lin.
\newblock The molecular characterization of weighted {H}ardy spaces.
\newblock {\em J. Funct. Anal.}, 188(2):442--460, 2002.

\bibitem[LLL09]{Lee-Lin-Lin}
Ming-Yi Lee, Chin-Cheng Lin, and Ying-Chieh Lin.
\newblock A wavelet characterization for the dual of weighted {H}ardy spaces.
\newblock {\em Proc. Amer. Math. Soc.}, 137(12):4219--4225, 2009.

\bibitem[LLY05]{LeeLinYang}
Ming-Yi Lee, Chin-Cheng Lin, and Wei-Chi Yang.
\newblock {$H^p_w$} boundedness of {R}iesz transforms.
\newblock {\em J. Math. Anal. Appl.}, 301(2):394--400, 2005.

\bibitem[LN15]{LernerNazarov}
Andrei~K. Lerner and Fedor Nazarov.
\newblock Intuitive dyadic calculus: the basics.
\newblock \url{http://arxiv.org/abs/1508.05639}, August 2015.
\newblock online.

\bibitem[Mey90]{Yves-ondelettes}
Yves Meyer.
\newblock {\em Ondelettes et op\'erateurs. {II}}.
\newblock Actualit\'es Math\'ematiques. [Current Mathematical Topics]. Hermann,
  Paris, 1990.
\newblock Op\'erateurs de Calder\'on-Zygmund. [Calder\'on-Zygmund operators].

\bibitem[MS13]{multilinear_harmonic}
Camil Muscalu and Wilhem Schlag.
\newblock {\em Classical and {Multilinear} {Harmonic} {Analysis}}.
\newblock Cambridge University Press, 2013.

\bibitem[MTT04]{MTTBiest0}
Camil Muscalu, Terence Tao, and Christoph Thiele.
\newblock {$L^p$} estimates for the biest. {I}. {T}he {W}alsh case.
\newblock {\em Math. Ann.}, 329(3):401--426, 2004.

\bibitem[Neu87]{weak-max-func}
Christoph~J. Neugebauer.
\newblock Iterations of {H}ardy-{L}ittlewood maximal functions.
\newblock {\em Proc. Amer. Math. Soc.}, 101(2):272--276, 1987.

\bibitem[Ste93]{SteinBigBook}
Elias~M. Stein.
\newblock {\em Harmonic analysis: real-variable methods, orthogonality, and
  oscillatory integrals}, volume~43 of {\em Princeton Mathematical Series}.
\newblock Princeton University Press, Princeton, NJ, 1993.
\newblock With the assistance of Timothy S. Murphy, Monographs in Harmonic
  Analysis, III.

\end{thebibliography}
